\documentclass[12pt, a4paper]{article}

\usepackage[T1]{fontenc}
\usepackage{amsmath}
\usepackage{amsfonts}
\usepackage{amssymb}
\usepackage{amsthm}
\usepackage{enumitem}
\usepackage[normalem]{ulem}
\usepackage[mathscr]{euscript}
\usepackage[utf8]{inputenc}
\usepackage{dsfont}
\usepackage{float}
\usepackage{times}
\usepackage{bbm}
\usepackage{xcolor}
\usepackage{xparse}
\usepackage[colorlinks,linkcolor=blue,citecolor=blue,urlcolor=blue]{hyperref}
\usepackage[capitalise]{cleveref}
\usepackage[left=3cm, right=3cm, bottom=2cm, top=2cm]{geometry}
\usepackage{tcolorbox}
\usepackage{tikz}
\usepackage{tikz-cd} 
\usetikzlibrary{calc}
\usetikzlibrary{arrows}
\usetikzlibrary{decorations.markings}
\usepackage{pgfplots}
\pgfplotsset{compat=1.17}

\newtheorem{theorem}{Theorem}[section]
\newtheorem{lemma}[theorem]{Lemma}
\newtheorem{proposition}[theorem]{Proposition}
\newtheorem{corollary}[theorem]{Corollary}

\theoremstyle{definition}
\newtheorem{assumption}[theorem]{Assumption}

\newtheorem{notation}[theorem]{Notation}
\newtheorem{definition}[theorem]{Definition}
\newtheorem{remark}[theorem]{Remark}

\theoremstyle{remark}

\numberwithin{equation}{section}

\crefname{example}{Example}{Examples}
\Crefname{example}{Example}{Examples}

\crefname{assumption}{Assumption}{Assumptions}
\Crefname{assumption}{Assumption}{Assumptions}

\crefname{condition}{Condition}{Conditions}
\Crefname{condition}{Condition}{Conditions}

\usepackage[framemethod=TikZ]{mdframed}
\newmdtheoremenv{ftheorem}[theorem]{Theorem}
\newmdtheoremenv{fdefinition}[theorem]{Definition}
\newmdtheoremenv{flemma}[theorem]{Lemma}

\setlist{topsep=1ex, itemsep=0.5ex, before={\setlist{topsep=-.5ex}}}
\setlist[itemize]{label=\textbullet}

\DeclareMathOperator{\trace}{tr}

\newcommand{\R}{\ensuremath{\mathbb{R}}}

\newcommand{\one}{{{\mathbbm 1}}}

\def\f{\frac}
\def\epsilon{\varepsilon}

\NewDocumentCommand{\Lip}{om}{\IfNoValueTF{#1}{|#2|_{\mathrm{Lip}}}{|#2|_{\mathrm{Lip};\,#1}}}

\def\E{\mathbb E}

\newcommand{\II}{\ensuremath{I \! \! I}}

\renewcommand{\geq}{\geqslant}
\renewcommand{\leq}{\leqslant}

\def\${|\!|\!|}

\def\<{\left\langle}
\def\>{\right\rangle}

\newcommand{\vertiii}[1]{{\left\vert\kern-0.25ex\left\vert\kern-0.25ex\left\vert #1
\right\vert\kern-0.25ex\right\vert\kern-0.25ex\right\vert}}
\newcommand{\rom}[1]{(\textup{\uppercase\expandafter{\romannumeral#1}})}
\def\sgn{{\mathop {\rm sign}}}

\makeatletter
\newcommand{\substackal}[1]{%
\vcenter{%
\Let@ \restore@math@cr \default@tag
\baselineskip\fontdimen10 \scriptfont\tw@
\advance\baselineskip\fontdimen12 \scriptfont\tw@
\lineskip\thr@@\fontdimen8 \scriptfont\thr@@
\lineskiplimit\lineskip
\ialign{\hfil$\m@th\scriptstyle##$&$\m@th\scriptstyle{}##$\hfil\crcr
#1\crcr
}%
}%
}
\makeatother

\newcommand{\p}{\partial}
\newcommand{\ps}{\frac{\partial}{\partial s}}
\newcommand{\pss}{\frac{\partial^2}{\partial s^2}}
\newcommand{\pt}{\frac{\partial}{\partial t}}
\newcommand{\Dt}{\frac{D}{d t}}
\newcommand{\Dps}{\frac{D}{\partial s}}
\newcommand{\Ds}{\frac{D}{d s}}
\newcommand{\DDt}{\frac{D^2}{dt^2}}

\newcommand{\la}{\left\langle}
\newcommand{\ra}{\right\rangle}
\newcommand{\lc}{\left(}
\newcommand{\rc}{\right)}
\newcommand{\li}{\left|}
\newcommand{\ri}{\right|}

\newcommand{\Ric}{\mathrm{Ric}}
\newcommand{\vol}{\mathrm{vol}}
\newcommand{\Hess}{\mathrm{Hess}}

\makeatletter
\newcommand{\newparallel}{\mathrel{\mathpalette\new@parallel\relax}}
\newcommand{\new@parallel}[2]{%
  \begingroup
  \sbox\z@{$#1T$}
  \resizebox{!}{\ht\z@}{\raisebox{\depth}{$\m@th#1/\mkern-5mu/$}}%
  \endgroup
}
\makeatother

\def\Xint#1{\mathchoice
{\XXint\displaystyle\textstyle{#1}}%
{\XXint\textstyle\scriptstyle{#1}}%
{\XXint\scriptstyle\scriptscriptstyle{#1}}%
{\XXint\scriptscriptstyle\scriptscriptstyle{#1}}%
\!\int}
\def\XXint#1#2#3{{\setbox0=\hbox{$#1{#2#3}{\int}$ }
\vcenter{\hbox{$#2#3$ }}\kern-.6\wd0}}

\def\dashint{\Xint-}

\definecolor{LB}{rgb}{0.29, 0.63, 0.73}

\begin{document}

\title{
Coarse Ricci curvature of weighted Riemannian manifolds}
\author{Marc Arnaudon \thanks{Univ. Bordeaux, CNRS, Bordeaux INP, IMB, UMR 5251,
F-33400 Talence, France. \newline  \texttt{marc.arnaudon@math.u-bordeaux.fr}}, Xue-Mei Li \thanks{ Dept. of Maths., Imperial College London, U.K.  \& EPFL, Switzerland \newline  \texttt{xue-mei.li@imperial.ac.uk} or  \texttt{xue-mei.li@epfl.ch} }, Benedikt Petko\thanks{Dept. of Maths.,  Imperial College London, U.K. \newline  \texttt{benedikt.petko15@imperial.ac.uk}}}
\date{\today}
\maketitle

\begin{abstract}
We show that the generalized Ricci tensor of a weighted complete Riemannian manifold can be retrieved asymptotically from a scaled metric derivative of Wasserstein 1-distances between normalized weighted local volume measures. As an application, we demonstrate that the limiting coarse curvature of random geometric graphs sampled from Poisson point process with non-uniform intensity converges to the generalized Ricci tensor. 
\end{abstract}


\tableofcontents

\section{Introduction and result}

In Riemannian geometry, the Riemann curvature tensor is closely related to the parallel transport of vectors along a curve, while the Ricci curvature is obtained by taking the trace of this tensor. The Ricci curvature is descriptive of properties of the manifold pertaining to volume. For example, the Bishop-Gromov comparison theorem relates global lower bounds of the Ricci curvature to the volume growth of geodesic balls on a manifold as the radius of a ball increases.


We briefly recall the standard objects playing a central role in the present work. A standard reference for geometric analysis is \cite{MR3726907}. Let $M$ be a complete Riemannian manifold of dimension $n$.  Denoting the $C^\infty(M)$-module of smooth vector fields as $\Gamma(TM)$, a linear connection is a bilinear map
$
\nabla: \Gamma(TM) \times \Gamma(TM) \rightarrow \Gamma(TM)$
satisfying the Leibniz rule, and should be interpreted as a rule for differentiating vector fields on the manifold. It is in one-to-one correspondence with parallel transport of vectors along curves as follows. For a given connection $\nabla$ and a smooth curve $t \mapsto \gamma_t$, we denote by $\newparallel_t: T_{\gamma_0}M \rightarrow T_{\gamma_t}M$ the map such that for every $v \in T_{\gamma_0}M$, $t \mapsto \newparallel_t v$ is the solution to the covariant differential equation
$$
\nabla_{\dot{\gamma}_t} (\newparallel_t v) = 0, \quad \newparallel_0 v = v,
$$
where we employ the commonly used abuse of notation by identifying $\nabla$ with the pull-back connection $\gamma^*\nabla$. This defining equation is commonly referred to as the parallel condition for $t \mapsto \newparallel_t v$ and $\newparallel_t$ is called the parallel transport map along $\gamma$.

Conversely, given a parallel transport $\newparallel_t: T_{\gamma_0}M \rightarrow T_{\gamma_t}M$ for any smooth curve $t\mapsto \gamma_t$, a linear connection can be defined for any $v \in T_{\gamma_0}M$ and smooth vector field $U$ on a neighbourhood of $x$ as
$$
\nabla_v U (x) = \lim_{t \rightarrow 0} \f{\newparallel_t^{-1} U(\gamma_t) - U(x)}{t}, \quad \dot{\gamma}_0 = v,
$$
where the limit takes place in the inner product space $T_xM$.

Throughout the article we employ the Levi-Civita connection, which is the unique metric and torsion free linear connection on the tangent bundle. The Riemann curvature tensor quantifies how much the vector field differentiation $(X,Y) \to \nabla_{X}\nabla_{Y}Z$
fails to commute. Precisely, the Riemann curvature tensor is defined for any three smooth vector fields $X, Y$ and $Z$ as
$$
R_{x}(X,Y)Z:=\nabla_X (\nabla_YZ)(x)- \nabla_Y(\nabla_XZ)(x),
$$
which also sets the sign convention for the Riemann curvature tensor for the remainder of this work.
Since this is a tensorial object, i.e. a $C^\infty(M)$-linear map, it is valid to define $R_x(u,v)w$ for any three vectors $u,v,w\in T_xM$. The Ricci curvature at a point $x$ is given by the expression
$$
\Ric_x(u,v):=-\trace R_x(u, \cdot)v = -\sum_{i=1}^n \< R_x(u,e_i)v, e_i\> = \sum_{i=1}^n \< R_x(u,e_i)e_i, v\> 
$$ 
for an arbitrary orthonormal basis $(e_i)_{i=1}^n$ of $T_xM$. For the last equality we applied the well-known antisymmetry of the Riemann tensor. In the sequel we mostly omit the basepoint superscript for tensor fields as the basepoint is usually clear from the context. 

The covariant derivative for a smooth vector field $Z$ can be modified by damping:
$\Ric- \nabla Z$.
This is known as the generalised Ricci curvature tensor and is related to conformally changing the Riemannian metric. In the classical case of $Z=0$, this condition was studied extensively in the renowned works \cite{Cheng-Yau75,Li-Yau86,Hamilton93,Ni-Tam}.
When $Z=-\nabla V$, i.e. we have a gradient field, global lower bounds on  $\Ric + 2 \Hess \; V$ are widely known as the Bakry-\'Emery criterion \cite{BE}, which provides the geometrically correct notion of global lower curvature bounds for weighted Riemannian manifolds. The Bakry-\'Emery criterion is equivalent to the so-called curvature-dimension condition for the operator $\f{1}{2}\Delta - Z$, which now has a vast literature building on it in more general settings,  
and especially on metric measure spaces \cite{Lott-Villani, Sturm06, Ambrosio-Gigli-Savare-14} and also on sub-Riemannian manifolds \cite{Juillet, Agrachev-Lee,MR3175142}. 
Less known is the variable bound version of the Bakry-\'Emery criterion, which was  successfully explored in \cite{Li,Wei-Wylie}.
In the context of stochastic analysis, the generalised Ricci tensor is closely related to damped parallel transport along Brownian paths which was studied by X.-M. Li, see \cite{MR4304478, MR3828180}. Denoting $\Ric^{\sharp}: TM\to TM$ as the bundle map uniquely associated with the Ricci tensor $\Ric: TM\times TM\to \R$ by the relation
$$
\<\Ric^\#(u),v \> = \Ric(u,v) \quad \forall u,v \in T_xM,
$$
and denoting by $(x_t)_{t \geq 0}$ the Brownian motion on $M$, 
the properties of the solution $t \mapsto W_t \in T_{x_t}M$ to the stochastic damped parallel equation
$$
\nabla_{\dot{\gamma}_t} W_t = -\f{1}{2} \Ric_{x_t}^\# W_t - \nabla_{W_t} \nabla V, \quad W_0 = \mathrm{Id},
$$
were used to obtain novel bounds for the Schr\"odinger semigroup and the Hessian of the fundamental solution of the Schr\"odinger operator.

The approach of characterizing the Ricci curvature using optimal transport distances of probability measures has been referred to in the literature as "coarse curvature" or "Wasserstein curvature" of the underlying space. The main motivation for considering such characterisation is that Wasserstein distances only require a metric structure to be defined, hence such notion can apply to spaces without a manifold structure. 

Specifically, we extend the approach of Ollivier \cite{OLLIVIER2007,MR2484937}. We define a coarse generalised Ricci curvature, presenting a rigorous proof which includes Ollivier's result as a special case.
The original definition of coarse Ricci curvature of Ollivier was motivated by Riemannian geometry, concretely by the observation that parallel geodesics starting from nearby points converge in positive sectional curvature planes and diverge in negative curvature planes, see \cite{MR3060504} for a visual introduction.

The advantage of Ollivier's notion of Wasserstein curvature is twofold. First, it can be explicitly computed on a number of examples, in particular graphs. Second, on Riemannian manifolds it allows for the approximation or retrieval of the Ricci curvature at each point. With this and the application to curvature of random geometric graphs in mind, we consider the case of a Riemannian manifold with a potential. As a novel element, we consider the Wasserstein distances of probability-normalized weighted volume measures on small geodesic balls to account for a smooth potential on the manifold and recover the smooth generalized Ricci curvature of Riemannian manifolds from a scaled metric derivative of Wasserstein distances of such measures.

The phenomenon underlying Wasserstein curvature is referred to in broader literature as Wasserstein contractivity, which was first established as a characterization of global lower Ricci curvature bounds on Riemannian manifolds in \cite{MR2142879}. Joulin \cite{MR2543873} defined the Wasserstein curvature in terms of the contraction ratio of the Wasserstein distances between Markov kernels. Further works on Wasserstein contractivity include e.g. \cite{MR2502429,Kuwada,Alfonsi-Corbetta-Jourdain}.

The Ricci curvature of a Riemannian manifold measures the expansion of the Riemannian volume, which is quantified by the Bishop-Gromov theorem. The relation between curvature, volume growth and optimal transport was also hinted at in \cite{Cordero-Erausquin-McCann-Schmuckenschlager, Ambrosio-Rigot,Brenier}.
Therefore, it is not surprising that Ollivier's definition of coarse Ricci curvature begins with a family of uniform volume measures indexed by points in the manifold. Such a family of measures can also be interpreted as a random walk.

Coarse Ricci curvature of random geometric graphs was studied in \cite{hoorn-2023}, where graphs generated by a Poisson point process with increasing intensity were considered. As a first result of this type for coarse Ricci curvature, the main result of their work is that under some assumptions on the relationship between sampling intensity, connectivity radius and random walk radius, the coarse curvature of random geometric graphs converges in $L^1$ at every point to the Ricci curvature of the underlying manifold. 

Our notion of generalised Ricci curvature extends the approach of Ollivier. We then follow up with an exploration of the concept of random graphs generated by connecting nearby points of a Poisson point process, as previously studied by van~der Hoorn, Cunningham,  Lippner, Trugenberger, and Krioukov \cite{hoorn-2023}. In our approach, we choose the weighted volume measures for the intensity measure of the Poisson process.

As a continuation of work on coarse curvature, the authors of the present article defined a notion of coarse extrinsic curvature in the setting of Riemannian submanifolds embedded in Euclidean spaces, and showed it to contain a meaningful extrinsic geometric information in the preprint \cite{arnaudon2024coarseextrinsiccurvatureriemannian}.

We briefly mention the literature on non-smooth curvature preceding the work of Ollivier. A notion of non-smooth sectional curvature bounds was first introduced by Alexandrov, beginning in the 1950s, see e.g. \cite{MR1185284} on the study of Alexandrov curvature lower bounds and more comprehensively the monograph \cite{MR1835418}. Non-smooth curvature has since been studied from various perspectives.
The entropy convexity approach to Ricci curvature lower bounds of metric measure spaces was studied in the seminal works \cite{Sturm2, Sturm1,Lott-Villani}.
We point out that Ollivier's approach in \cite{OLLIVIER2007,MR2484937} described differs from entropy convexity in that it approximates the value of the Ricci curvature at every point, rather than describing only a Ricci curvature global lower bound. Ollivier's coarse curvature on graphs was further studied in \cite{Lin-Yau, MR2872958,Bauer-Jost-Liu12, MR3164168}. An alternative notion of non-smooth curvature on graphs termed "Ricci flatness" was proposed in \cite{MR1426537}, presented in more detail also in the monograph \cite{MR1421568}. The concept of coarse curvature was also extended to non-commutative transportation cost in \cite{Gao-Rouze}. Recently, coarse curvature has found applications in manifold learning and community detection, see e.g. \cite{Ache-Warren, Wu-Hu-Yu, Sia-Jonckheere-Bogdan}.

\subsection*{Main results}

For any $x \in M$ and $\varepsilon >0$ we denote the uniform probability measure supported on the geodesic ball $B_\varepsilon(x)$ as
$$
\begin{aligned}
d\mu^\varepsilon_x(z) &:= \f{\one_{B_\varepsilon(x)}(z)}{\vol(B_\varepsilon(x))} d\vol(z),
\end{aligned}
$$
where $\vol$ is the standard Riemannian volume measure.

Let $V: M \rightarrow \R$
be a smooth potential and $e^{-V(z)}d\vol(z)$ the corresponding weighted measure on $M$.  For any $x \in M$, define the (non-uniform) probability measure supported on the geodesic ball $B_\varepsilon(x)$,
$$
\begin{aligned}
d\nu_x^\varepsilon(z) &:= 
\one_{B_\varepsilon(x)}(z) \f{e^{-V(z)}}{\int_{B_\varepsilon(x)} e^{-V(z')}d\vol(z')}d\vol(z).
\end{aligned}
$$
Our main results are \cref{thm:0} and \cref{graph-continuum-limit-curvature}. The first allows to extract the generalized Ricci tensor from 1-Wasserstein distances of such measures:
\begin{theorem}
\label{thm:0}
For any point $x_0 \in M$, vector $v \in T_{x_0}M$ with $\|v\| =1$, sufficiently small $\delta, \varepsilon > 0$ and $y := \exp_{x_0}(\delta v)$, it holds that
\begin{equation}
\label{eq:expansion-0}
W_1(\nu_{x_0}^\varepsilon, \nu_y^\varepsilon) = \delta\left(1-\f{\varepsilon^2}{2(n+2)}
\left(\Ric_{x_0}(v,v)+2\Hess_{x_0}V(v,v)\right)\right) 
+ O(\delta^2 \varepsilon^2) + O(\delta \varepsilon^3).
\end{equation}
\end{theorem}

Denoting the coarse curvature at scale $\varepsilon$ as
\begin{equation}
\kappa_\varepsilon(x_0,y) := 1 -\f{W_1(\nu_{x_0}^\varepsilon,\nu_y^\varepsilon)}{d(x_0,y)},
\end{equation}
we deduce upon rearrangement of the expansion (\ref{eq:expansion-0}) and taking the limit that
\begin{equation}
\label{coarse-ricci}
\lim_{\varepsilon, \delta \rightarrow 0} \f{2(n+2)}{\varepsilon^2} \kappa_\varepsilon(x_0,y) = 
\Ric_{x_0}(v,v)+2 \Hess_{x_0}V(v,v).
\end{equation}
We present a detailed proof of \cref{thm:0}, which extends the result of Ollivier \cite[Example 7 \& Section 8]{MR2484937}.

As an application, in Theorem \ref{graph-continuum-limit-curvature} we  extend a result of Hoorn et al. \cite{hoorn-2023}, which showed that the coarse curvature of random geometric graphs sampled from a Poisson point process with increasing intensity, proportional to the non-uniform measure $e^{-V(z)} \vol(dz)$, converges to the smooth Ricci curvature modified by the Hessian of $V$. Our method allows to deal with the non-uniformity of the intensity of the Poisson process as well as the non-uniformity incurred by the exponential mapping.

\begin{remark}
We qualify the term "sufficiently small" for $\delta, \varepsilon$ in Theorem \ref{thm:0}. In all arguments, we will assume that $\delta, \varepsilon$ are as small as needed in a way only dependent on a compact neighbourhood of $x_0 \in M$. The need for such restriction is for two reasons:
\begin{itemize}
\item in manifold distance estimates of Section \ref{section:manifold-distance-estimates} to
ensure that all geodesics in the variations used are length-minimizing and unique. This can
be done by restricting $\delta, \varepsilon$ to some small enough fraction of the uniform
injectivity radius at $x_0$,
\item in Wasserstein distance estimates of Section \ref{section:wasserstein-approximations}, to apply the Inverse Function Theorem for the transport map $T$ which has non-zero determinant at $x_0$.
\end{itemize}
We will assume throughout this work that such restrictions are in place and are covered implicitly by the "sufficiently small" assumption for $\delta, \varepsilon$.
\end{remark}

\begin{remark}
Ollivier \cite[Example 5]{MR2484937} presented a similar result by using uniform measures shifted in the direction of $-\nabla V(x)$ for a uniform measure centered at $x \in M$. Nonetheless, to obtain this, the correct magnitude for this shift is chosen a posteriori. This is in contrast to our method of non-uniform measures, which does not allow this degree of freedom and can therefore be seen as more intrinsic to the weighted manifold. We emphasize the distinction in that our method employs non-uniform measures rather than uniform. Moreover, our approach is more suitable for our application to random geometric graphs sampled from a Poisson point process with non-uniform intensity.
\end{remark}

We now follow through with two ingredients needed for proving \cref{thm:0}, manifold distance and Wasserstein distance estimates, presented in Sections \ref{section:manifold-distance-estimates} and \ref{section:wasserstein-approximations}, respectively. The application to random geometric graphs constitutes \cref{rgg-application}. The clearly presented intermediate geometric estimates could be of independent interest.

\bigskip

{\bf  Acknowledgement.} This research has been supported by the EPSRC Centre for Doctoral Training in Mathematics of Random Systems: Analysis, Modelling and Simulation (EP/S023925/1). XML acknowledges partial support from the EPSRC (EP/S023925/1 and EP/V026100/1).

\section{Manifold distance estimates}
\label{section:manifold-distance-estimates}

The 1-Wasserstein distance of measures is by definition the minimum of an average of distances of pairs of points on the manifold. Variations of geodesics are a standard tool for local estimation of distances on Riemannian manifolds. We introduce preliminary notation for this section.

\begin{notation}
\label{notation-2}
Let $v \in T_{x_0}M$ be a unit vector. The maps $c: [0, \epsilon] \times [0,1] \rightarrow M$ will denote various smooth variations, to be specified, of the geodesic $\gamma:[0,1]\rightarrow M$,
$$
\gamma(t) := \exp_{x_0}(t\delta v)
$$
so that $c(0,t) = \gamma(t)$. Denote the Jacobi field corresponding to the variation $c$ as
$$
J(t) := \left.\frac{\partial}{\partial s}\right|_{s=0}c(s,t)
$$
and the covariant derivatives of $J$ with respect to $\dot{\gamma}$ along $\gamma$ by $\Dt J(t), \f{D^2}{dt^2}J(t)$. Recall $J$ satisfies the Jacobi equation
$$
\DDt J(t) = - R(J(t),\dot{\gamma}(t))\dot{\gamma}(t),
$$
where $R$ is the Riemann curvature tensor. For any vector field $X$ defined along $\gamma$, we denote the part perpendicular to $\dot{\gamma}$ as 
$$
X^\perp := X - \f{1}{\delta^2}\< X, \dot{\gamma}\> \dot{\gamma}.
$$
Denote by $\newparallel_t  (\gamma): T_{x_0}M \rightarrow T_{\gamma(t)}M$ the parallel transport along the geodesic $\gamma$ with respect to the Levi-Civita connection. When there is no risk of confusion, we shall write $\newparallel_t$ for $\newparallel_t (\gamma)$, omitting mentioning the geodesic $\gamma$.

\end{notation}

We shall denote by $\|\cdot\|$ the Riemannian norm of tangent vectors. For the distance estimates we will make use of the standard formulas for first and second derivatives of the length 
$$
L(\varepsilon, \delta) := \int_0^1 \left\|\frac{\partial c(\varepsilon,t)}{\partial t}\right\| dt
$$
in terms of $J$ (see \cite[Chap. 6]{MR3726907}):
\begin{lemma}
The first two derivatives of the length $L$ in the first variable at $0$ are
\begin{align}
\label{eq:variation-order-1}
\left.\ps\right|_{s=0} L(s, \delta) &= \frac{1}{\delta} \left.\la J(t),\dot{\gamma}(t)
\ra\right|_{t=0}^{t=1},\\
\left.\pss\right|_{s=0}L(s, \delta) &= \f{1}{\delta}  \left[ \int_0^1 \left( \la \Dt J(t)^\perp,\Dt
J(t)^\perp \ra - \la R(J(t)^\perp,\dot{\gamma}(t))\dot{\gamma}(t), J(t)^\perp \ra \right) dt \right] \nonumber \\
& \qquad +\f{1}{\delta}\left.\la \Dps \left.\ps\right|_{s=0} c(s,t),\dot{\gamma}(t)
\ra \right|_{t=0}^{t=1}
\nonumber\\ \label{eq:variation-order-2}
&= \f{1}{\delta} \left.\la \Dt J(t)^\perp,J(t)^\perp \ra \right|_{t=0}^{t=1} +
\f{1}{\delta}\left.\la \Dps \left.\ps\right|_{s=0} c(s,t),\dot{\gamma}(t)\ra
\right|_{t=0}^{t=1}.
\end{align}
\end{lemma}

We will consider two different variations $c_1$ and $c_2$ of the geodesic $\gamma(t) = \exp_{x_0} (t\delta v)$ which provide key distance estimates for the optimal transport problem in Theorem \ref{thm:0}. 

\subsection{Pointwise transport distance estimate}

We proceed with defining $c_1$, the purpose of which is to approximate pointwise transport distance by a certain transport map defined later. Using these pointwise distance estimates, we will be able to conclude an upper bound for the Wasserstein distance in \cref{thm:0}.

\begin{notation}
\label{transport-vector}
For any $v, w \in T_{x_0}M$ with $\|v\|=1$, $\|w\| \leq 1$ and the geodesic $\gamma=\exp_{x_0}(t \delta v)$,  we introduce the transport vector
\begin{equation}
\label{w-dash}
w' := w - \frac{\varepsilon}{2}(1-\|w\|^2) ( \newparallel_1^{-1}(\gamma) \; \nabla V(y)-\nabla V(x_0)) = w + O(\delta
\varepsilon),
\end{equation}
where we abuse the notation $O(\delta \varepsilon)$ to denote a vector of magnitude of order $\delta \varepsilon$.\\
Denote the geodesics
\begin{equation}
\label{eta-theta-def}
\theta(s) := \exp_{y} (s \newparallel_1 w'), \quad \eta(s) :=
\exp_{x_0} (sw),
\end{equation}
where $y=\exp_{x_0}(\delta v)$. Define the map 
$c_1: [0, \varepsilon] \times [0,1] \rightarrow M$ as
$$
c_1(s,t) = \exp_{\eta(s)} (t \exp^{-1}_{\eta(s)}(\theta(s))).
$$
This represents a family of geodesics indexed by $s$ and parametrized by $t$, starting from $\eta(s)$ and reaching $\theta(s)$ at $t=1$. In particular, it is a variation of the geodesic $t \mapsto \exp_{x_0}(t\delta v)$. See \cref{fig:c1} for an illustration of this variation. Note that 
$$
\theta(s) = c_1(s,1), \quad \eta(s) = c_1(s,0)
$$
are the bottom and the top curves in \cref{fig:c1} and for every $s \in [0,\varepsilon]$, $t \mapsto c_1(s,t)$ is a geodesic. Moreover, $\gamma(t)= \exp_{x_0}(t \delta v) = c_1(0,t)$ is the leftmost geodesic in the figure.
\end{notation}

The aim of this variation is to estimate the distance of the two corners $\theta(\varepsilon)$ and $\eta(\varepsilon)$ given by the variation $c_1$ of the geodesic $\gamma$ from $\eta$ to $\theta$. To emphasize the dependency on both $\delta$ and $\varepsilon$, we will denote the length
\begin{equation}
\label{length-1}
L_1(\varepsilon, \delta) := \int_0^1 
\left\|\frac{\partial c_1(\varepsilon,t)}{\partial t}\right\| dt
= d(c_1(\varepsilon,1),c_1(\varepsilon,0)) = d(\theta(\varepsilon),\eta(\varepsilon)).
\end{equation}
We begin with preliminary estimates on the Jacobi field of the variation $c_1$.
\begin{lemma}
  \label{jacobi-field-estimates}
  The Jacobi field $J(t)=\left.\ps\right|_{s=0}c_1(s,t)$ satisfies
  \begin{equation}
  \label{jacobi-field-estimate-1}
  \left\|\Dt J(t)\right\| = O(\delta \varepsilon)+O(\delta^2) \quad \forall t \in [0,1]
  \end{equation}
  and
  \begin{equation}
  \label{jacobi-field-estimate-2}
  \newparallel_t^{-1} J(t)= J(0) + O(\delta \varepsilon) + O(\delta^2) \quad \forall t \in [0,1].
  \end{equation}
\end{lemma}
\begin{proof}
As $\gamma(t) = \exp_{x_0}(t\delta v)$ is a constant speed geodesic starting in the direction of the unit vector $v$, we have $\dot{\gamma}(t)=\delta \newparallel_t v$. Hence the Jacobi equation implies
\begin{equation}
  \label{3}
\left\| \DDt J(t)\right\|= \|R(J(t),\dot{\gamma}(t))\dot{\gamma}(t)\| \leq C \delta^2 \|J(t)\|,
\end{equation}
where $C = \sup_{t \in [0,1]} \|R_{\gamma(t)}(\cdot, \newparallel_t v)\newparallel_t v\| < \infty$ in the sense of operator norm. The expansion of $\newparallel_t^{-1} J(t)$ at $t=1$ in $T_{x_0}M$ with integral remainder is
$$
\newparallel_1^{-1}J(1) = J(0) + \left.\f{d}{dt}\right|_{t=0} \newparallel_t^{-1} J(t) + \int_0^1 (1-u) \f{d^2}{du^2} \newparallel_u^{-1} J(u)du.
$$
Applying the relation $\f{d^k}{dt^k} \newparallel_t^{-1} = \; \newparallel_t \f{D^k}{dt^k}$ for every $k\in \mathbb{N}$ and $\newparallel_0 \; = \textrm{Id}$, we obtain in terms of covariant derivatives:
\begin{equation}
\label{jacobi-field-expansion}
\newparallel_1^{-1} J(1) = J(0) + \left.\Dt\right|_{t=0} J(t) + \int_0^1 (1-u)\newparallel_u^{-1}\frac{D^2}{du^2} J(u)du.
\end{equation}
The Jacobi field satisfies the boundary conditions
\begin{equation}
\label{initial-final-jacobi}
J(0)= \f {d}{ds}\Big|_{s=0} \eta(s)=w , \qquad J(1)= \f
{d}{ds}\Big|_{s=0}\theta(s)= \; \newparallel_1 w',
\end{equation}
where $w'$ is given by \eqref{w-dash}. Together with
$$
\newparallel_1^{-1}J(1)-J(0) = w'-w = O(\delta\varepsilon),
$$
the expansion \eqref{jacobi-field-expansion} implies
$$
\left\|\left.\Dt\right|_{t=0} J(t)\right\| \leq O(\delta \varepsilon) + C\delta^2 \|J(t)\|.
$$
Therefore,
$$
\|J(t)\| \leq \|J(0)\| + \int_0^t\left\|\frac{D}{du} J(u)\right\|du \leq 1 + O(\delta \varepsilon) + C\delta^2 \int_0^t \|J(u)\|du,
$$
which yields by Gr\"onwall's lemma that
$$
\|J(t)\| \leq (1 + O(\varepsilon \delta))e^{C\delta^2 t} = 1+O(\delta^2)+ O(\delta \varepsilon).
$$
As a consequence, we deduce from (\ref{3}) that
$
\DDt J(t) = O(\delta^2)
$,
and thus
$$
\newparallel_t^{-1} \Dt J(t) = \left.\Dt\right|_{t=0} J(t) + \int_0^t \newparallel_u^{-1}\frac{D^2}{du^2} J(u)du = O(\delta \varepsilon) + O(\delta^2).
$$
Finally,
$$
\newparallel_t^{-1} J(t) = J(0) + \int_0^t \newparallel_u^{-1} \f{D}{du} J(u)du,
$$
where the second term has norm of the required order.
\end{proof}

The following fact about order of magnitude can be extended to arbitrary number of variables, but the two-variable version will suffice for our purpose.
\begin{lemma}
\label{smooth-joint-order}
    Let $\phi: \R^2 \rightarrow \R$ be a smooth function. For any $a,b \in \mathbb{N}$, if $\phi(x,y)=O(\min(x^a, y^b))$ then $\phi(x,y)=O(x^a y^b)$ in a fixed neighbourhood of the origin.

\end{lemma}
\begin{proof}
The multivariate Taylor's theorem states that
\begin{equation}
\label{multivariate-taylor}
\phi(x,y) = \sum_{i+j\leq a+b-1} \partial_x^i \partial_y^j \phi(0,0)x^i y^j + \sum_{i+j = a + b} r_{ij}(x,y)x^i y^j
\end{equation}
for smooth remainders $r_{ij}: \R^2 \rightarrow \R$.
Then the assumption on $\phi$ implies in particular
$$
\phi(x,y) = O(x^a) \quad \forall y \in \R,
$$
which is equivalently stated as $\limsup_{x\rightarrow 0} \left| \f{\phi(x,y)}{x^a} \right| < \infty \quad \forall y \in \R$.
We claim that this implies
$$
\partial_x^i \partial_y^j \phi(0,0) = 0, \quad r_{ij}(x,y)=O(x^{a-i}) \quad \forall i \leq a-1, j \in \mathbb{N}, y \in \R.
$$
Indeed, if $\partial_x^i \partial_y^j \phi(0,0) \neq 0$  for some $i \leq a-1$, then 
\begin{equation}
\label{contradiction:1}
\infty = \lim_{x\rightarrow 0} \left| \f{\partial_x^i \partial_y^j \phi(0,0)x^i y^j}{x^a} \right| \leq \limsup_{x\rightarrow 0} \left| \f{\phi(x,y)}{x^a} \right|  \quad \forall y \neq 0.
\end{equation}
Similarly, if a remainder $r_{ij}$ for some $j \leq b-1$ satisfies $\limsup_{x \rightarrow 0} \left| \f{r_{ij}(x,y)}{x^{a-i}} \right| = \infty$ then
\begin{equation}
\label{contradiction:2}
\infty = \lim_{x\rightarrow 0} \left|\f{r_{ij}(x,y)x^iy^j}{x^a}\right| \leq \limsup_{x\rightarrow 0} \left| \f{\phi(x,y)}{x^a} \right| \quad \forall y \neq 0.
\end{equation}
Either of \eqref{contradiction:1} and \eqref{contradiction:2} contradicts that $\phi(x,y) = O(x^a)$ for all $y \in \R$.

In the same way, it holds that $\phi(x,y) = O(y^b)$ for all $x \in \R$, which implies
$$
\partial_x^i \partial_y^j \phi(0,0) = 0, \quad r_{ij}(x,y)=O(y^{b-j}) \quad \forall i \in \mathbb{N}, j \leq b-1, x \in \R.
$$
Hence the first sum in \eqref{multivariate-taylor} vanishes.
By smoothness, also $r_{ab}(x,y)=O(1)$ in a fixed neighbourhood of the origin.
As a consequence, we can write
$$
\begin{aligned}
\phi(x,y) &= \sum_{i+j = a+b} r_{ij}(x,y)x^iy^j \\
&= \sum_{i=0}^{a-1} r_{i,a+b-i}(x,y) x^i y^{a+b-i} + r_{ab}(x,y)x^ay^b + \sum_{j=0}^{b-1} r_{a+b-j,j}(x,y) x^{a+b-j} y^{j}\\
&= \sum_{i=0}^{a-1} O(x^{a-i}) x^iy^{a+b-i} + O(x^ay^b) +\sum_{j=0}^{b-1} O(x^{a+b-j}) y^{a+b-j} y^{j}\\
&= O(x^ay^{b+1}) + O(x^a y^b) + O(x^{a+1}y^b) = O(x^a y^b),
\end{aligned}
$$
as required.
\end{proof}

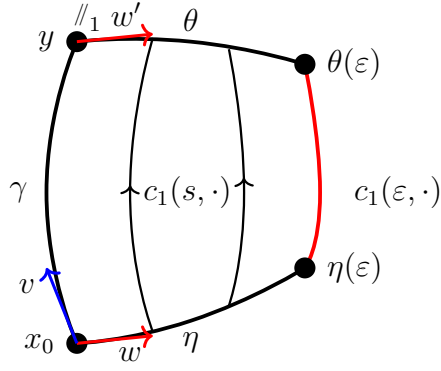
\begin{figure}
\centering
\begin{tikzpicture}
\draw [line width=0.5mm] plot [smooth, tension=1] coordinates { (0,0) (-0.4,2) (0,4) };
\node at (-0.5,2) [left] {$\gamma$};
\draw [line width=0.5mm] plot [smooth, tension=1] coordinates { (0,0) (1.5, 0.3) (3,1) };
\node at (1.5, 0.3) [below] {$\eta$};
\node[circle, fill=black, scale=0.7, label=left:{$x_0$}] at (0,0) {};
\draw [line width=0.5mm, red] plot [smooth, tension=1] coordinates { (3,1) (3.2,2) (3,3.7) };
\node at (3.5,2) [right] {$c_1(\varepsilon,\cdot)$};
\draw [line width=0.5mm] plot [smooth, tension=1] coordinates { (0,4) (1.5,4) (3,3.7) };
\node at (1.5,4) [above] {$\theta$};
\node[circle, fill=black, scale=0.7, label=right:{$\theta(\varepsilon)$}] at (3,3.7) {};
\node[circle, fill=black, scale=0.7, label=right:{$\eta(\varepsilon)$}] at (3,1) {};
\node[circle, fill=black, scale=0.7, label=left:{$y$}] at (0,4) {};

\begin{scope}[very thick,decoration={
    markings,
    mark=at position 0.5 with {\arrow{>}}}
    ] 
\draw [line width=0.3mm, postaction={decorate}] plot [smooth, blue, tension=1] coordinates {
(1,0.15) (0.7,2) (1,4.05) };
\draw [line width=0.3mm, postaction={decorate}] plot [smooth, blue, tension=1] coordinates {
(2,0.5) (2.2,2) (2, 3.9) };
\node at (2.2,2) [left] {$c_1(s,\cdot)$};
\end{scope}

\draw [line width=0.4mm,->, color=red] (0,0) -- (1,0.1);
\node at (1,0.1) [below left] {$w$};
\draw [line width=0.4mm,->,color=blue] (0,0) -- (-0.4,1);
\node at (-0.4,1) [below left] {$v$};
\draw [line width=0.4mm,->, color=red] (0,4) -- (1,4.1);
\node at (1,4) [above left] {$\newparallel_1 w'$};

\end{tikzpicture}
\caption{Geodesic variation $c_1$ (with positive sectional curvature in the $v,w$-plane)}
\label{fig:c1}
\end{figure}

For linearly independent $v,w \in T_{x_0}$, denote by 
$$
K_{x_0}(v,w) := \frac{\<R(v,w)w,v\>}{\|w\|^2\|v\|^2 - \<v,w\>^2}
$$
the sectional curvature at $x_0$ of the tangent plane spanned by $v$ and $w$.

\cref{variation-1} below is similar to the classical distance expansion by means of a geodesic triangle, which was the original characterization of sectional curvature by Riemann, see e. g. \cite{Meyer2004ToponogovsTA} for the proof.
\begin{lemma}
\label{triangle-distance}
For any $w_1,w_2 \in T_{x_0}M$ linearly independent and $\varepsilon >0$ sufficiently small, 
$$
d(\exp_{x_0}(\varepsilon w_1), \exp_{x_0}(\varepsilon w_2)) = \varepsilon \|w_1-w_2\|-\frac{1}{6}\f{\<R(w_1,w_2)w_2,w_1 \>}{\|w_1-w_2\|}\varepsilon^3 + O(\varepsilon^4).
$$
\end{lemma}

The following contrasts with the triangle estimate of \cref{triangle-distance} in that the two geodesics start from two distinct points and have a carefully chosen relationship between their initial directions $w$ and $\newparallel_1 w'$, prescribed by \eqref{w-dash}.

\begin{proposition}
\label{variation-1}
For any $v,w \in T_{x_0}M$ with $\|v\|=1, \|w\| \leq 1$ and $\delta, \varepsilon$ sufficiently small, we have the estimate for the distance between $\exp_{x_0}(\varepsilon w)$ and $\exp_{\exp_{x_0}(\delta v)}(\varepsilon \newparallel_1 w')$ where $w'$ is given by (\ref{w-dash}), expressed by the geodesic length
\begin{equation}
\label{c2-length-expansion}
\begin{aligned}
L_1(\varepsilon,\delta)
&=\delta \left( 1- \frac{\varepsilon^2}{2} \left[K_{x_0}(v,w)(\|w\|^2-\<v, w\>^2)+
\Hess_{x_0} V(v,v) (1 - \|w\|^2)\right]\right)\\
&\qquad +O(\delta^2 \varepsilon^2 ) + O(\delta \varepsilon^3),
\end{aligned}
\end{equation}
where the $O(\varepsilon^2\delta^2)+O(\delta \varepsilon^3)$ terms are uniformly bounded in $v$ and $w$.
\end{proposition}

\begin{proof}
We expand the length $L_1(\varepsilon, \delta)$  in the first variable,
\begin{equation}
\label{L1-expansion}
L_1(\varepsilon,\delta) = L_1(0,\delta) + \varepsilon \left.\ps\right|_{s=0}
L_1(s,\delta)+ \f{\varepsilon^2}{2} \left.\pss \right|_{s=0} L_1(s,\delta) +
O(\varepsilon^3)
\end{equation}
and compute the first and second order coefficients using the variation $c_1(s, \cdot)$.

Recalling the boundary conditions \eqref{initial-final-jacobi} for the Jacobi field $J(t) = \left. \ps \right|_{s=0} c_1(s,t)$, the first order coefficient in the expansion (\ref{L1-expansion}) is
\begin{equation}
\label{variation-1-first-derivative}
\begin{aligned}
\left.\ps\right|_{s=0} L_1(s,\delta) &= \f{1}{\delta} (\< J(1), \dot{\gamma}(1) \> - \<
J(0), \dot{\gamma}(0)
\>)\\
&= \<\newparallel_1 w', \newparallel_1 v\> - \< w, v \>\\
&= \< \newparallel_1 w - \frac{\varepsilon}{2}(1-\|w\|^2) ( \nabla V(y)- \newparallel_1 \; \nabla V(x_0)), \newparallel_1 v \> - \<w,v\> \\
&= - \frac{\varepsilon}{2}(1-\|w\|^2) \< \newparallel_1^{-1} \; \nabla V(y) - \nabla V(x_0), v \>\\
&= -\frac{\delta \varepsilon}{2}(1-\|w\|^2) \Hess_{x_0}V(v,v) + O(\delta^2 \varepsilon).
\end{aligned}
\end{equation}
Here we used the formula (\ref{eq:variation-order-1}) on the first line, inserted \eqref{initial-final-jacobi} on the second line, plugged in \eqref{w-dash} for $w'$ on the third line, applied the isometry $\newparallel_1^{-1}$ on the fourth line, and expanded 
$$
\newparallel_1^{-1} \; \nabla V(y) - \nabla V(x_0) = \delta \Hess_{x_0}V(v,v) + O(\delta^2)
$$
on the last line. 

The second order variation of length formula (\ref{eq:variation-order-2}) reduces to
$$
\left.\pss\right|_{s=0} L_1(s,\delta) = \f{1}{\delta}  \int_0^1 \left( \la \Dt J(t)^\perp,\Dt
J(t)^\perp \ra - \< R(J(t)^\perp,\dot{\gamma}(t))\dot{\gamma}(t), J(t)^\perp \> \right) dt,
$$
since $\Dps \ps c(s,0) = \Dps \ps c(s,1) = 0$ as $s \mapsto c(s,0), s \mapsto c(s,1)$ are geodesics. By the estimate (\ref{jacobi-field-estimate-1}), the first term is
\begin{equation}
\label{second-order-term-1}
\int_0^1 \la \Dt J(t)^\perp, \Dt J(t)^\perp \ra dt = O(\delta^2 \varepsilon^2) +
O(\delta^3 \varepsilon) + O(\delta^4).
\end{equation}
Moreover, by smoothness of $R$ in the base point, we have $\|\newparallel_t^{-1}\circ R \circ \newparallel_t - R\| = O(\delta)$ as the parallel transport is along a geodesic of length $\delta$. Then by multi-linearity of $R$ and the estimate (\ref{jacobi-field-estimate-2}):
\begin{equation}
\label{second-order-term-2}
\begin{aligned}
\int_0^1 \< R(J(t)^\perp,\dot{\gamma}(t))\dot{\gamma}(t), J(t)^\perp \> dt &= \delta^2 \int_0^1 \<R(\newparallel_t^{-1} J(t)^\perp, v) v, \newparallel_t^{-1}J^\perp(t)\>dt+O(\delta^3)\\
&= \delta^2 \<R(w,v)v,w\> +O(\delta^3),
\end{aligned}
\end{equation}
noting that $ \<R(w^\perp,v)v,w^\perp\> = \<R(w,v)v,w\>$ by anti-symmetry of the curvature tensor and since $w^\perp-w$ is parallel to $v$ by definition.

Therefore, coming back to the second order coefficient and plugging in \eqref{second-order-term-1} and \eqref{second-order-term-2},
\begin{equation}
\label{variation-1-second-derivative}
\left.\pss\right|_{s=0} L_1(s,\delta) =
-\delta \<R(w,v)v,w\> + O (\delta^2)
= -\delta K(v,w)(\|w\|^2-\<v,w\>^2) + O(\delta^2).
\end{equation}
Moreover, the $O(\varepsilon^3)$ term in (\ref{L1-expansion}) is in fact $O(\delta \varepsilon^3)$ by \cref{smooth-joint-order} since $ L_1(\varepsilon, \delta) = O(\delta)$ and $L_1$ is smooth. Finally, we obtain the expansion by plugging \eqref{variation-1-first-derivative} and \eqref{variation-1-second-derivative} back into \eqref{L1-expansion}. We emphasize that the Hessian manifests itself through \eqref{variation-1-first-derivative}. The vector $w'$ was carefully chosen so that the Hessian appeared at the correct order together with the Ricci curvature.
\end{proof}

\subsection{Projection distance estimate}
We construct the variation $c_2$ for the purpose of estimating the projection distance to a specific submanifold. This projection distance will serve as a suitable 1-Lipschitz function for establishing a lower bound on the Wasserstein distance in \cref{thm:0} by the Kantorovich-Rubinstein duality. 

Let $E$ be a smooth embedded submanifold of the Riemannian manifold $M$.
\begin{definition}
    An open neighbourhood $U$ of $E$ in $M$ is said to be a tubular neighbourhood if there exists an open subset of the normal bundle $W \subset TE^\perp$ such that  $\exp : W \rightarrow U$ is a diffeomorphism and $E\subset \exp W$. 
\end{definition}

The following is a regularity result for the projection map and projection distances, see one of \cite{MR0473443,MR0614221,MR749908} for a proof. 
\begin{lemma}
\label{signed-distance}
For every compact, embedded smooth submanifold $E$ of $M$, there is a tubular neighbourhood $U$ of $E$ in $M$ such that the shortest distance projection map
\begin{equation}
\label{projection}
p(x) := \mathrm{argmin}_{z \in E} d(x,z)
\end{equation}
is well-defined and smooth on $U$, and the distance to projection $z \mapsto d(z,p(z))$ is smooth on $U \setminus E$.
Moreover, if $E$ is a codimension 1 submanifold with $\nu \in \Gamma(TE^\perp)$ a unit vector field normal to the submanifold then the signed distance to projection defined by
\begin{equation}
\label{signed-projection-distance-def}
f(z) := \sgn (\langle \exp_{p(z)}^{-1} (z), \nu (z) \rangle) d(z,p(z))
\end{equation}
is smooth on all of $U$.
\end{lemma}

Recall $v \in T_{x_0}M$ is the fixed unit vector, representing the direction of transport of the initial test measure $\mu^\varepsilon_{x_0}$. We now consider the concrete submanifold of codimension 1,
\begin{equation}
\label{submanifold}
E := \exp_{x_0} \{v^\perp\} := \{\exp_{x_0}(w) : w \in T_{x_0}M, \langle w, v \rangle = 0\}
\subset M.
\end{equation}
 We recall that the second fundamental form for $E$ is the 2-covariant tensor field, defined at $x_0$ for two vectors $w_1,w_2 \in T_{x_0}E$ as
\begin{equation}
\label{sff-definition}
\II_{x_0}(w_1,w_2) := \< \nabla^M_{w_1} \nu(x_0), w_2 \> = - \< \nu, \nabla^M_{w_1} W (x_0)\>,
\end{equation}
where $W$ on the right is an arbitrary smooth local tangent vector field on $E$ with $W(x_0)=w_2$ and $\nu$ a smooth local normal vector field on $E$ with $\nu(x_0)=v$, see \cite[Chap. 5]{MR3726907}. Here $\nabla^M$ refers to the Levi-Civita connection on $M$.

\begin{remark}
\label{vanishing-sff}
We shall make use of the fact that $\II$ vanishes at $x_0 \in M$ for this submanifold, which we prove for the readers' convenience. Using the normal coordinates $(x^1,\ldots,x^{n})$ of $M$ at $x_0 \in M$, we may assume $(x^1,\ldots,x^{n-1})$ are the normal coordinates of $E$ at $x_0$. Then $\Gamma^k_{ij}(x_0) = 0$ (see \cite[Chap. 1.4]{MR3726907}) and for arbitrary vector fields $X=\sum_{i=1}^{n-1} X^i \f{\p}{\p x^i},Y=\sum_{i=1}^{n-1} Y^i \f{\p}{\p x^i} \in \Gamma(TE)$, it holds that
$$
\begin{aligned}
\nabla_X Y(x_0) &= \sum_{i,j=1}^{n-1} X^i(x_0) \Gamma^k_{ij}(x_0) Y^j(x_0)
\frac{\partial}{\partial x^k}(x_0)
+ X^i(x_0) \frac{\partial Y^j}{\partial x^j}(x_0) \frac{\partial}{\partial x^j}(x_0) \\
&= \sum_{i,j=1}^{n-1} X^i(x_0) \frac{\partial Y^j}{\partial x^j}(x_0)
\frac{\partial}{\partial x^j}(x_0) \in T_{x_0}E.
\end{aligned}
$$
We conclude $\II(X,Y)(x_0) = -\langle \nu(x_0), \nabla_X Y(x_0) \rangle=0$ since $\nu$ is normal to $E$.
\end{remark}

\begin{figure}
\centering
\includegraphics[scale=0.4]{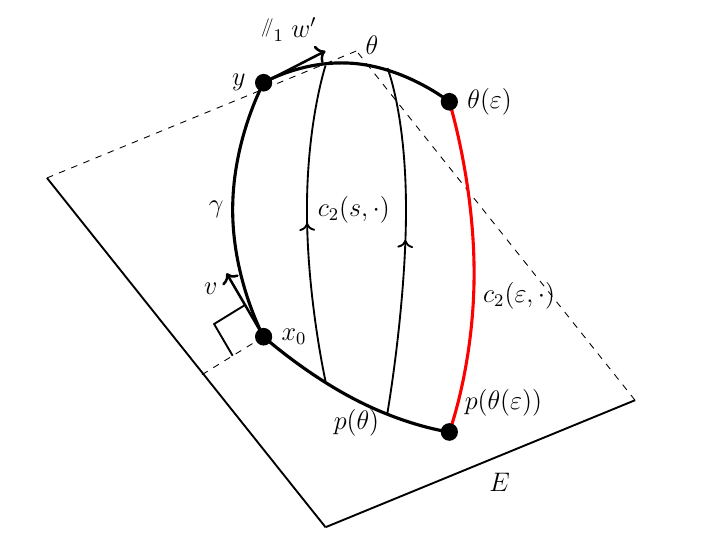}
\caption{Geodesic variation $c_2$}
\label{fig:c2}
\end{figure}

We now define the geodesic variation $c_2$, with its depiction in \cref{fig:c2}. As before, denote for a vector $w \in T_{x_0}M$ with $\|w\| \leq 1$ the geodesic
$$
\theta(s) := \exp_{y} (s \newparallel_1 w'),
$$
where $w'$ is defined by (\ref{w-dash}). Let $c_2: [0, \varepsilon] \times [0,1] \rightarrow M$ be the variation of the geodesic
$\gamma(t) = \exp_{x_0} (t\delta v)$ defined by 
$$
c_2(s,t) = \exp_{p(\theta(s))} (t \exp^{-1}_{p(\theta(s))}(\theta(s))),
$$
with $p$ the projection map given by (\ref{projection}). Since $\theta$, $p$, $\exp_{x_0}$ and $\exp_{x_0}^{-1}$ are all smooth (in particular $p$ by \cref{signed-distance}), $c_2$ is also smooth. Note also the boundary value $c_2(s,0) = \exp_{p(\theta)}(\mathbf{0}) = p(\theta)$.
Denote the length
$$
L_2(\varepsilon, \delta) := \int_0^1 \left\|\frac{\partial
c_2(\varepsilon,t)}{\partial t}\right\| dt
= d(c_2(\varepsilon, 1), c_2(\varepsilon, 0))
= d(\theta(\varepsilon),p(\theta(\varepsilon)))
$$
and the signed length
\begin{equation}
\label{L2-tilde}
\tilde{L}_2(\varepsilon, \delta) := \sgn(\<\exp_{x_0}^{-1}(c_2(\varepsilon,\delta)),v\>) L_2(\varepsilon, \delta),
\end{equation}
which is smooth by \cref{signed-distance}.

\begin{proposition}
\label{prop:variation-2}
For any vectors $v,w \in T_{x_0}M$ with $\|v\|=1, \|w\| \leq 1$ and $\delta, \varepsilon$ sufficiently small, the signed distance between $\exp_{\exp_{x_0}(\delta v)}(\varepsilon \newparallel_1 w')$ and its projection to $E$ has expansion 
\begin{equation}
\begin{aligned}
\tilde{L}_2(\varepsilon, \delta)
&= \delta + \varepsilon \<v, w\> - \frac{\varepsilon^2 \delta}{2}
\left[K_{x_0}(v,w)(\|w\|^2 - \<v,w\>^2)+
\Hess_{x_0} V(v,v) (1 - \|w\|^2)\right]\\
& \qquad + O(\varepsilon^2 \delta^2) + O(\varepsilon^3),
\end{aligned}
\end{equation}
where the $O(\varepsilon^2\delta^2)+O( \varepsilon^3)$ terms are uniformly bounded in $v$ and $w$.
\end{proposition}

\begin{proof}
Since $\tilde{L}_2(s,\delta)=L_2(s,\delta)$ for small $s$, their derivatives at $s=0$ also agree, and we may apply again the method of geodesic variations to compute $\left.\ps\right|_{s=0} L_2(s, \delta)$ and $\left.\pss\right|_{s=0} L_2(s,\delta)$ in order to obtain the coefficients of the expansion of $\tilde{L}_2(\varepsilon,\delta)$.

The Jacobi field $J(t)=\left.\ps\right|_{s=0} c_2(s,t)$ satisfies boundary condition $J(1) = \; \newparallel_1 w'$. Moreover, $J(0) \in T_{x_0}E$ because $c_2(s,0) \in E$ for all $s \in [0,\varepsilon]$, hence in particular $\dot{\gamma}(0) \perp J(0)$ and $J^\perp(0) = J(0)$. By the variation of length formula (\ref{eq:variation-order-1}) and plugging in \eqref{w-dash} for $w'$, the order $\varepsilon$ coefficient is thus
\begin{align}
\left.\ps\right|_{s=0} L_2(s, \delta) &=
\f{1}{\delta}(\<J(1),\dot{\gamma}(1)\>-\<J(0),\dot{\gamma}(0)\>) \nonumber \\
&= \f{1}{\delta}\<\newparallel_{1}^{-1}J(1), \newparallel_{1}^{-1} \dot{\gamma}(1)\> \nonumber \\
&= \<w', v\> \\
\label{L2-first-derivative}
&= \<w,v\> - \f{\delta \varepsilon}{2} \Hess_{x_0}(v,v)(1-\|w\|^2).
\end{align}
By the formula (\ref{eq:variation-order-2}), the coefficient for the $\varepsilon^2$ term is
\begin{equation}
\label{L2-second-derivative}
\begin{aligned}
\left.\pss\right|_{s=0} L_2(s,\delta) &=
\f{1}{\delta}\left(\<\left.\Dt\right|_{t=1} J^\perp(t), J^\perp(1)\>
-\<\left.\Dt\right|_{t=0} J^\perp(t), J^\perp(0)\>\right)\\
&\quad +\f{1}{\delta} \left(\la \Dps \left.\ps\right|_{s=0} c_2(s,1),\dot{\gamma}(1)\ra -\la
\Dps \left.\ps\right|_{s=0} c_2(s,0),\dot{\gamma}(0)\ra\right).
\end{aligned}
\end{equation}
We show that all terms on the right vanish except for the first one which we then further estimate. First, since $\Dt \dot{\gamma}(t)$ = 0 and $\gamma$ is a constant speed geodesic, we note that
$$
\Dt J^\perp (t) = \Dt \left( J(t) - \f{\<J(t),\dot{\gamma}\>}{\|\dot{\gamma}\|^2}\dot{\gamma}(t) \right)= \Dt J(t) - \f{\<\Dt J(t), \dot{\gamma}(t)\>}{\|\dot{\gamma}(t)\|}\dot{\gamma}(t) = \left(\Dt J(t)\right)^\perp.
$$
Then for any $u \in T_{x_0}E$:
$$
\begin{aligned}
\< \left.\Dt\right|_{t=0} J^\perp(t), u \> &= \< \left( \left.\Dt\right|_{t=0} J(t) \right)^\perp, u \> \\
&= \< \lc \left.\f{D}{\p s} \pt \right|_{t=s=0} c_2(s,t) \rc^\perp, u \> \\
&= \< \left.\f{D}{\p s} \pt \right|_{t=s=0} c_2(s,t) , u \>,
\end{aligned}
$$
using the torsion-free property of the Levi-Civita connection on the second line, and using $u \perp \dot{\gamma}(0)$ on the third line by definition of $E$. To bring in the second fundamental form, in terms of notation for its definition \eqref{sff-definition}, the unit normal vector field on $E$ is
$\nu(c_2(s,0)) = \f{1}{\delta} \left. \pt \right|_{t=0}c_2(s,t)$ and the two submanifold directions are $w_1 = \left.\ps\right|_{s=0} c(s,0)$ and $w_2=u$. Then we can write 
$$
\nabla^M_{w_1} \nu(x_0) = \f{1}{\delta} \left.\Dps \pt \right|_{s=t=0} c_2(s,0),
$$
and thus
$$
\< \left.\f{D}{\p s} \pt \right|_{t=s=0} c_2(s,t) , u \> = \delta \< \nabla^M_{w_1} \nu (x_0), w_2 \> = \delta \II_{x_0}(w_1,w_2) = 0,
$$
because the second fundamental form vanishes at $x_0$ as per \cref{vanishing-sff}. Hence, 
$$
\forall u \in T_{x_0}E: \< \left.\Dt\right|_{t=0} J^\perp(t), u \> =0,
$$
which together with $\Dt J^\perp(0) = \left(\Dt J(0)\right)^\perp \in T_{x_0}E$ implies $\left.\Dt\right|_{t=0} J^\perp (t) =  0$. 

Similarly, with the same notation matching to \eqref{sff-definition} additionally with $W(c_2(s,0)) := \ps c(s,0)$ so that $W(x_0) = W(c_2(0,0)) =w_1$,
$$
\< \Dps \left.\ps\right|_{s=0} c(s,0),\dot{\gamma}(0)\> = \delta \<\nabla^M_{w_1} W(x_0), \nu(x_0)\> =- \delta \II_{x_0}(w_1,w_1) = 0.
$$
Since $s \mapsto \theta(s)$ is a geodesic, $\Dps \ps c(s,1) = \Ds \theta(s)= 0$, the third term is also 0. 

We now estimate the first term in (\ref{L2-second-derivative}). For any $t \in [0,1]$ we have the expansion with integral remainder,
$$
\begin{aligned}
\newparallel_t^{-1} J^\perp(t) &= J^\perp(0)+\int_0^t (t-u)\newparallel_u^{-1}
\DDt J^\perp(u) du \\
&=J^\perp(0)-\delta^2\int_0^t (t-u) \newparallel_u^{-1} R(J^\perp(u), \newparallel_u v)\newparallel_u v \; du \\
&=J^\perp(0)+O(\delta^2).
\end{aligned}
$$
This in particular also holds for $t=1$ so we deduce that
$$
\newparallel_1^{-1} J^\perp(1) - \newparallel_t^{-1}J^\perp(t) = O(\delta^2).
$$
Moreover, pulling back the Jacobi equation back to $T_{x_0}M$ by $\newparallel_1^{-1}$ and integrating over $t\in [0,1]$, using $\left.\Dt \right|_{t=0}J(t)=0$ as shown above, gives
$$
\newparallel_1^{-1} \left.\Dt\right|_{t=1} J^\perp(t) = - \delta^2 \int_0^1 \newparallel_t^{-1} R(J^\perp(t), \newparallel_t v)\newparallel_t v \; dt.
$$
Therefore, the second derivative is
$$
\begin{aligned}
\left.\pss\right|_{s=0} L_2(s,\delta) &=
\f{1}{\delta} \<\left.\Dt\right|_{t=1} J^\perp(t), J^\perp(1)\>\\
&=-\delta \int_0^1 \<\newparallel_t^{-1} R(J^\perp(t), \newparallel_t v) \newparallel_t v, \newparallel_1^{-1} J^\perp(1)\>dt\\
&=-\delta \int_0^1 \< R(\newparallel_t^{-1}J^\perp(t), v)v, \newparallel_1^{-1} J^\perp(1)\> dt +O(\delta^2)\\
&=-\delta \int_0^1 \< R(\newparallel_1^{-1}J^\perp(1)+O(\delta^2), v)v, \newparallel_1^{-1} J^\perp(1)\> dt +O(\delta^2)\\
&=-\delta \< R(w', v)v, w'\> +O(\delta^2)\\
&=-\delta \< R(w, v)v, w\>+O(\delta^2),
\end{aligned}
$$
using that $\|\newparallel_t^{-1} \circ R \circ \newparallel_t- R\|=O(\delta)$ and $w'=w+O(\delta \varepsilon)$.

Hence the coefficients of the expansion of $\tilde{L}_2(\varepsilon,\delta)$ in $\varepsilon$ are as required.
\end{proof}

We may deduce the expansion in $\varepsilon$ of the difference of signed lengths:
\begin{corollary}
\label{difference-signed-lengths}
For any vectors $v,w \in T_{x_0}M$ with $\|v\|=1, \|w\| \leq 1$ and $\delta, \varepsilon$
sufficiently small, 
$$
\begin{aligned}
\tilde{L}_2(\varepsilon, \delta) &- \tilde{L}_2(\varepsilon, 0)\\
&= 
\delta \left(1- \frac{\varepsilon^2}{2} \left[K_{x_0}(v,w)(\|w\|^2 - \<v,w\>^2) + \Hess_{x_0}
V(v,v) (1 - \|w\|^2)\right]\right)\\
& \qquad + O(\delta^2 \varepsilon^2) + O(\delta \varepsilon^3)\\
&=L_1(\varepsilon,\delta)+O(\delta^2 \varepsilon^2)+O(\delta\varepsilon^3),
\end{aligned}
$$
where the $O(\varepsilon^2\delta^2)+O(\delta \varepsilon^3)$ terms are uniformly bounded in $v$ and $w$.
\end{corollary}
\begin{proof}
From the preceding proposition, subtracting the two lengths we first obtain almost the expansion above, except the last term is at first only $O(\varepsilon^3)$, but since the difference vanishes is also $O(\delta)$, it must in fact be $O(\delta \varepsilon^3)$, see \cref{smooth-joint-order}. The second equality in the statement of the corollary follows by comparison with \cref{variation-1} since the expansions agree up to the term $O(\delta^2 \varepsilon^2)+O(\delta\varepsilon^3)$.
\end{proof}

\section{Wasserstein distance approximations}
\label{section:wasserstein-approximations}

The following two preliminary lemmas relate the Ricci curvature to the sectional curvature by integral averages, thereby providing a bridge between the distance estimates of the previous section and the Wasserstein distance estimates that will follow.

The standard Ricci curvature is usually defined as the contraction of the Riemann curvature tensor and expressed equivalently by the sectional curvature,
$$
\forall v \in T_{x_0}M: \Ric(v,v) := \sum_{i=1}^n \la R(v, e_i)e_i, v \ra = \sum_{i=1}^n K(v,e_i)(\|v\|^2-\<v,e_i\>^2)
$$
for an arbitrary orthonormal basis $(e_i)_{i=1}^n$ of $T_{x_0}M$.
The Ricci curvature can equivalently be expressed as an average over a sphere or ball of arbitrary non-zero radius. We shall denote by $B_r$ the ball of radius $r>0$ in $T_{x_0}M$ centered at the origin and $\sigma$ the uniform surface measure on the sphere $\partial B_\varepsilon$. The proofs of the following two elementary items can be found in the appendix.

\begin{lemma}[Ricci curvature as average over a sphere]
\label{ricci-as-integral-over-sphere}
For any $x_0 \in M, v \in T_{x_0}M$ with $\|v\|=1$ and $\varepsilon >0$, it holds that
\begin{equation}
\label{eq:ricci-sphere-avg}
\frac{\varepsilon^2}{n} \Ric(v,v) = \dashint_{\partial B_\varepsilon} K(v,w)(\varepsilon^2- \la v,w \ra^2) d\sigma(w).
\end{equation}
\end{lemma}

We identify $T_{x_0}M$ with $\R^n$ in a standard way by matching an arbitrary orthonormal basis of $T_{x_0}M$ with another arbitrary orthonormal basis of $\R^n$, and denote by $dw$ the volume on $T_{x_0}M$ induced by any such identification.

\begin{lemma}[Ricci curvature as average over a ball]
\label{lemma:ricci-average-ball}
For any $x_0 \in M$,  $v \in T_{x_0}M, \|v\|=1$ and $\varepsilon >0$ it holds that
\begin{equation}
\begin{aligned}
\frac{\varepsilon^2}{n+2} \Ric(v,v)
&= \dashint_{B_\varepsilon} K(v,w) \lc \|w\|^2 - \la v,w \ra^2 \rc dw.
\end{aligned}
\end{equation}
\end{lemma}

\subsection{Density estimates}
\label{density-estimates}
\cref{lemma:ricci-average-ball} represented the Ricci curvature as an average over a ball in the tangent space. Since it's more natural to consider measures on the manifold, we will apply Lemma \ref{lemma:density} to follow to compare uniform measures on geodesic balls in the manifold and the push-forwards of the uniform measure on the ball in the tangent space by the exponential
map. The density estimates of this section are auxilliary in proving the Wasserstein distance approximations of the next section.

Denote the inverse exponential map on a small enough neighbourhood $B_{\varepsilon_0}(x_0)$ as
$$
\exp_{x_0}^{-1}: B_{\varepsilon_0}(x_0) \subset M \rightarrow T_{x_0} M
$$
and its derivative as
$$
D \exp_{x_0}^{-1}: TM|_{B_{\varepsilon_0}(x_0)} \rightarrow T_{x_0} M.
$$
The notation $D\exp_{x_0}^{-1}(z)(w)$ will mean the evaluation at the base point $z \in B_\varepsilon(x_0)$ and tangent vector $w \in T_zM$.

In normal coordinates at $x_0$,
$g_{ij}(x_0) = \delta_{ij}$ and $\Gamma^i_{jk}(x_0)=0$ (\cite[Chap. 1.4]{MR3726907}) and the equation for geodesics
$
\ddot{x}^i(t) + \Gamma^{i}_{jk}(x(t))\dot{x}^j(t) \dot{x}^k(t) = 0
$
implies $\ddot{x}^{i}(0)=0$. Then for an initial vector $v$ with $\|v\|=t < \varepsilon$, the integral remainder formula gives
$$
\exp_{x_0}^i(v) = x^i(t) = v^i + \f{1}{2} \int_0^t \dddot{x}^i(s)(t-s)^2 ds = v^i + O(\varepsilon^3),
$$
and consequently $\det D \exp_{x_0}(v) = 1 + O(\varepsilon^2)$. Upon inversion we may finally deduce the well-known fact:

\begin{lemma}
\label{inverse-exponential-expansion}
For all sufficiently small $\varepsilon >0$, it holds that
$$
\sup_{z \in B_\varepsilon(x_0)} \det (D \exp_{x_0}^{-1}(z)) = 1+O(\varepsilon^2).
$$

\end{lemma}

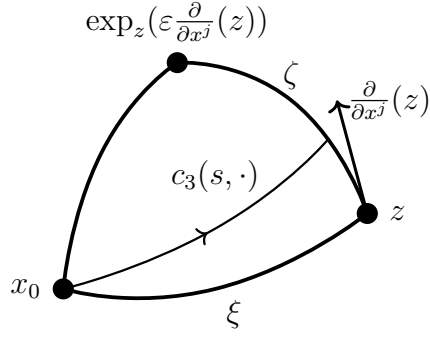
\begin{figure}
\centering
\begin{tikzpicture}
\draw [line width=0.5mm] plot [smooth, tension=1] coordinates { (0,0) (2,0) (4,1) };
\node[circle, fill=black, scale=0.7, label=left:{$x_0$}] at (0,0) {};
\node[circle, fill=black, scale=0.7, label=right:{$z$}] at (4,1) {};
\node[circle, fill=black, scale=0.7, label=above:{$\exp_z(e_j(1))$}] at (1.5,3) {};
\node at (2, 0) [below right] {$\xi$};

\draw [line width=0.5mm] plot [smooth, tension=1] coordinates { (0,0) (0.5,1.8) (1.5,3) };

\draw [line width=0.5mm] plot [smooth, tension=1] coordinates { (4,1) (3,2.5) (1.5,3) };
\node at (3, 2.5) [above] {$\zeta$};

\node at (2, 1.5) [] {$c_3(s,\cdot)$};

\draw [line width=0.4mm,->] (4,1) -- (3.6,2.5);
\node at (3.6,2.7) [right] {$e_j(1)$};

\begin{scope}[very thick,decoration={
    markings,
    mark=at position 0.5 with {\arrow{>}}}
    ] 
\draw [line width=0.3mm, postaction={decorate}] plot [smooth, blue, tension=1] coordinates {
(0,0) (2,0.8) (3.5,2) };
\end{scope}

\end{tikzpicture}
\caption{Geodesic variation $c_3$}
\label{fig:0}
\end{figure}

\begin{notation}
\label{notation:3}
For $x \in M$, denote by
$$
\tilde{B}_\varepsilon(x) \subset T_{x}M, \quad 
B_\varepsilon(x) := \exp_{x}\lc \tilde{B}_\varepsilon(x) \rc \subset M
$$
the $\varepsilon$-balls in $T_{x}M$ and $M$, respectively, and the probability measures
$$
\begin{aligned}
d\tilde{\mu}^\varepsilon_x(w) &= \f{\one_{\tilde{B}_\varepsilon(x)}(w)}
{|\tilde{B}_\varepsilon(x)|} dw & \mathrm{on } \; T_xM,\\
\bar{\mu}_{x}^\varepsilon(z) &= (\exp_x)_*\tilde{\mu}_x^\varepsilon(z)& \mathrm{on } \; M,\\
d\mu^\varepsilon_x(z) &= \f{\one_{B_\varepsilon(x)}(z)}{\vol(B_\varepsilon(x))} d\vol(z) & \mathrm{on } \; M,\\
d\tilde{\nu}_x^\varepsilon(w) &= 
\f{e^{-V(\exp_x w)}}{\int_{\tilde{B}_\varepsilon(x)} e^{-V(\exp_x w')}d\tilde{\mu}_x^\varepsilon(w')}d\tilde{\mu}_x^\varepsilon(w) & \mathrm{on } \; T_xM,\\
\bar{\nu}_x^\varepsilon(z)&= (\exp_x)_*\tilde{\nu}_x^\varepsilon(z) & \mathrm{on } \; M,\\
d\nu_x^\varepsilon(z) &= 
\f{e^{-V(z)}}{\int_{B_\varepsilon(x)} e^{-V(z')}d\mu_x^\varepsilon(z')}d\mu_x^\varepsilon(z) &\mathrm{on } \; M.
\end{aligned}
$$
\end{notation}

The auxilliary measures $\bar{\mu}_x^\varepsilon, \bar{\nu}_x^\varepsilon$ will be instrumental in approximating
$W_1(\nu_{x_0}^\varepsilon,\nu_y^\varepsilon)$ (\cref{wasserstein-approximations}). 

\begin{definition}
Let $(M,g)$, $(\tilde{M}, \tilde{g})$ be Riemannian manifolds of dimension $n$ and $F: M \rightarrow \tilde{M}$ a diffeomorphism. The Jacobian at a point $x \in M$ is
$$
\det DF(x) := \det (\<DF(x)(e_i), \tilde{e}_j\>)
$$
for arbitrary orthonormal bases $(e_j)_{j=1}^m$ and $(\tilde{e}_j)_{j=1}^n$ of $T_xM$ and $T_{F(x)}\tilde{M}$.
\end{definition}

\begin{lemma}
\label{det-local-coords}
Let $(x^1,\ldots,x^n)$ be local coordinates at $x \in M$ and $(\tilde{x}^1,\ldots,\tilde{x}^n)$ local coordinates at $F(x) \in \tilde{M}$. Then for all $z$ in the domain of $(x^1,\ldots,x^n)$ such that $F(z)$ is in the domain of $(\tilde{x}^1,\ldots,\tilde{x}^n)$,
$$
\begin{aligned}
\det DF(z) &= \det \left( (g^{-\f{1}{2}}(z))^\ell_i \f{\partial (\tilde{x}^k \circ F)}{\partial x^\ell}(z)(\tilde{g}^{\f{1}{2}}(F(z)))^j_k \right) \\
&= \det ((g^{-\f{1}{2}}(z))^i_j) \det \left(\f{\partial (\tilde{x}^i \circ F)}{\partial x^j} \right) \det ((\tilde{g}^{\f{1}{2}}(F(z)))^i_j).
\end{aligned}
$$ 
\end{lemma}

\begin{proof}
The coordinate vector fields can be orthonormalized using the square root of the respective metrics as
$$
e_j(z) := (g^{-\f{1}{2}}(z))^i_j  \f{\partial}{\partial x^i}(z), \quad j=1,\ldots,n,
$$
and
$$
\tilde{e}_j(F(z)) := (\tilde{g}^{-\f{1}{2}}(F(z)))^i_j  \f{\partial}{\partial \tilde{x}^i}(F(z)), \quad j=1,\ldots,n.
$$
Moreover, the derivative in terms of coordinate vector and covector fields is
$$
DF(z) = \f{\partial (\tilde{x}^k \circ F)}{\partial x^\ell}(z) \f{\partial}{\partial \tilde{x}^k}(F(z)) \otimes dx^\ell(z).
$$
Then
$$
DF(x)(e_i) = (g^{-\f{1}{2}}(z))^\ell_i \f{\partial (\tilde{x}^k \circ F)}{\partial x^\ell}(z) \f{\partial}{\partial \tilde{x}^k}(F(z))
$$
and
$$
\begin{aligned}
\<DF(x)(e_i), \tilde{e}_j(F(z))\> &= \tilde{g}_{kq}(F(z)) (g^{-\f{1}{2}}(z))^\ell_i \f{\partial (\tilde{x}^k \circ F)}{\partial x^\ell}(z) (\tilde{g}^{-\f{1}{2}}(F(z)))^q_j \\
&= (g^{-\f{1}{2}}(z))^\ell_i \f{\partial (\tilde{x}^k \circ F)}{\partial x^\ell}(z) (\tilde{g}^{\f{1}{2}}(F(z)))^j_k
\end{aligned}
$$
and the form of the determinant follows. It can be split into a product of determinants by usual rules of linear algebra.
\end{proof}

\begin{lemma} 
\label{lemma:density}
For $\delta, \varepsilon$ sufficiently small, $\bar{\mu}_{x}^\varepsilon$ and $\mu^\varepsilon_{x}$ are equivalent measures on $M$, supported on $B_\varepsilon(x)$ with mutual density
\begin{equation}
\label{eq:density}
\frac{d \bar{\mu}_{x}^\varepsilon}{d\mu_{x}^\varepsilon}(z) =  1 + h(z, x),
\end{equation}
where $h: M \times M \rightarrow \R$ is an a. e. smooth function such that
$h(\cdot,x)=O(\varepsilon^2)$ for every $x \in M$ and $\int h(z,x)d\mu_x(z)=0$.
\end{lemma}

\begin{proof} Let $(\tilde{x}^1,\ldots,\tilde{x}^n)$ be Euclidean coordinates of $T_xM$ and $(x^1, \ldots, x^n)$ the normal coordinates at $x_0 \in M$.
The uniform measures on $\tilde{B}_\varepsilon(x)$ and $B_\varepsilon(x)$ can be represented as differential forms,
\begin{equation}
\label{def:measures}
d\tilde{\mu}_{x}^\varepsilon :=
\tilde{C} \mathbbm{1}_{B_\varepsilon(0)} 
d\tilde{x}^1 \wedge \ldots \wedge d\tilde{x}^n,
\quad d\mu_{x}^\varepsilon := C \mathbbm{1}_{B_\varepsilon^M(x)} 
 \sqrt{\det(g_{ij})} 
dx^1 \wedge \ldots \wedge dx^n,
\end{equation}
where $\tilde{C} := \li \tilde{B}_\varepsilon(x) \ri^{-1}$ and $C := \mathrm{vol}(B_\varepsilon(x))^{-1}$ are the normalizing constants. Denoting $(\exp_x^{-1})^i=\tilde{x}^i\circ \exp_x^{-1}$, the push-forward $\bar{\mu}_{x}^\varepsilon$ is then
\begin{equation}
\label{mu-pushforward}
\begin{aligned}
d\bar{\mu}_{x}^\varepsilon(z) = (\exp_{x})_* d\tilde{\mu}_{x}^\varepsilon(z) 
&= \tilde{C} \mathbbm{1}_{B_\varepsilon(x)}(z)d(\tilde{x}^1 \circ \exp_{x}^{-1}) \wedge \ldots \wedge d(\tilde{x}^n \circ \exp_{x}^{-1}) \\
&= \tilde{C} \mathbbm{1}_{B_\varepsilon(x)}(z) \det \lc \frac{\partial (\tilde{x}^i \circ \exp^{-1}_{x})}{\partial x^j} (z)\rc dx^1 \wedge \ldots \wedge dx^n.
\end{aligned}
\end{equation}
See \cite[Lemma 9.11]{MR2954043} for a proof of the pushforward formula for differential forms. Comparing (\ref{def:measures}) and (\ref{mu-pushforward}),
since the determinants are non-vanishing we see that the
two measures are equivalent and their mutual density is given by
\begin{equation}
\label{eq:13}
\begin{aligned}
\frac{d \bar{\mu}_{x}^\varepsilon}{d\mu_{x}^\varepsilon}(z) &= \frac{\tilde{C}}{C} \det\lc \frac{\partial (\tilde{x}^i \circ \exp_{x}^{-1})}{\partial x^j} (z)\rc \times  \det(g_{ij}(z))^{-\frac{1}{2}} \\
&= \frac{\tilde{C}}{C} \det \left(\left((g^{-\f{1}{2}}(z))^i_k \frac{\partial (\tilde{x}^k \circ \exp_{x}^{-1})}{\partial x^j}(z)\right)_{ij} \right) \\
&= \frac{\tilde{C}}{C} \det D\exp_{x_0}^{-1}(z) \\
&= \frac{\tilde{C}}{C} (1+O(\varepsilon^2)),
\end{aligned}
\end{equation}
having applied \cref{det-local-coords} with $F= \exp_{x_0}^{-1}$ and $\tilde{g}_{ij}=\delta_{ij}$ on the penultimate line and \cref{inverse-exponential-expansion} on the last line.

It remains to see the ratio of the normalizing constants satisfies $\frac{\tilde{C}}{C} = 1 + O(\varepsilon^2)$ as a smooth function of the point $x \in M$. Integrating both sides of (\ref{eq:13}) over $B_\varepsilon(x)$ with respect to $\mu_{x}^\varepsilon$ we obtain
\begin{equation}
\label{eq:7}
\bar{\mu}_{x}^\varepsilon(B_\varepsilon(x)) = \frac{\tilde{C}}{C}
\mu_{x}^\varepsilon(B_\varepsilon(x)) (1+O(\varepsilon^2)).
\end{equation}
The probability measures $\bar{\mu}_{x}^\varepsilon, \mu_{x}^\varepsilon$ are both supported on $B_\varepsilon(x)$ so $\frac{\bar{\mu}_{x}^\varepsilon(B_\varepsilon^M(x))}{\mu_{x}^\varepsilon(B_\varepsilon^M(x))}=1$ and rearrangement of (\ref{eq:7}) gives $\frac{C}{\tilde{C}} = 1 + O(\varepsilon^2)$ as a smooth function of $x$. We obtain the same for the inverse ratio $\frac{\tilde{C}}{C}$ by the expansion $\frac{1}{1+z} = 1 - z + O(z^2)$. 

\noindent Since all factors in (\ref{eq:13}) are $1+O(\varepsilon^2)$, the same holds for their product, i.e.
\begin{equation}
\label{eq:density-factorization}
\frac{d\bar{\mu}_{x}^\varepsilon}{d\mu_{x}^\varepsilon}(u) 
=1+O(\varepsilon^2).
\end{equation}
Labelling the $O(\varepsilon^2)$ term as $h(\cdot,x)$, we have 
\begin{equation}
  \label{eq:22}
1 = \int \frac{d \bar{\mu}_{x}^\varepsilon}{d\mu_{x}^\varepsilon}(z) d\mu_{x}^\varepsilon(z)
= 1 + \int h(z, x) d\mu_{x}^\varepsilon(z),
\end{equation}
implying that $\int h(z,x) d\mu_{x}^\varepsilon(z)=0$. 
\end{proof}

\begin{lemma}
\label{lemma:density-mu-nu}
For any $x \in M$, the mutual densities expand as
$$
\begin{aligned}
\f{d\bar{\nu}_x^\varepsilon}{d\bar{\mu}_x^\varepsilon}(z) &=\one_{B_{\varepsilon}(x)}(z)
(1-\<\nabla V(x), \exp^{-1}_x(z) \> + h_1(z,x)),\\
\f{d\nu_x^\varepsilon}{d\mu_x^\varepsilon}(z) &=\one_{B_{\varepsilon}(x)}(z)
(1-\<\nabla V(x), \exp^{-1}_x(z) \> + h_2(z,x)),
\end{aligned}
$$
where $h_i: M \times M \rightarrow \R$ are a.e. smooth functions such that
$h_i(\cdot,x)=O(\varepsilon^2)$ for every $x \in M$ and $\int h_i(z,x)d\bar{\mu}^\varepsilon_x(z)=0$ for $i=1,2$.
\end{lemma}
\begin{proof}
From \cref{notation:3}, 
\begin{equation}
\label{dnu}
\f{d\bar{\nu}_x^\varepsilon}{d\bar{\mu}_x^\varepsilon}(z) = \f{d (\exp_x)_*\tilde{\nu}_x^\varepsilon}{d (\exp_x)_* \tilde{\mu}_x^\varepsilon}(z) = \f{d \tilde{\nu}_x^\varepsilon}{d\tilde{\mu}_x^\varepsilon}(\exp_x^{-1}z) = \f{e^{-V(z)}}{\int_{B_\varepsilon(x)} e^{-V(z)}d\bar{\mu}_x^\varepsilon(z)}.
\end{equation}
Expand the numerator in \eqref{dnu} as 
$$
\begin{aligned}
e^{-V(z)} &= e^{-V(x)} e^{V(x)-V(z)}\\
&=e^{-V(x)}(1-\<\nabla V(x), \exp^{-1}_x(z) \>+O(\varepsilon^2)).
\end{aligned}
$$
Similarly for the denominator,
$$
\begin{aligned}
e^{-V(x)}
\int_{B_\varepsilon(x)} (1-\<\nabla
V(x), \exp^{-1}_x(z') \>+O(\varepsilon^2)) d\bar{\mu}_x^\varepsilon(z')
= e^{-V(x)}\left(1+ O(\varepsilon^2) \right),
\end{aligned}
$$
where the $\nabla V$ term vanishes because
\begin{equation}
\label{eq:10}
\int_{B_\varepsilon(x)} \< \nabla V(x),\exp^{-1}_x(z') \> d\bar{\mu}_x^\varepsilon(z') = 
\langle \nabla V(x), \int_{B_\varepsilon(0)} w d\tilde{\mu}_x^\varepsilon(w)\rangle = 0.
\end{equation}
Cancelling the factor $e^{-V(x)}$, the density (\ref{dnu}) can thus be written as
$$
\f{d\bar{\nu}_x^\varepsilon}{d\bar{\mu}_x^\varepsilon}(z) = \f{1-\<\nabla V(x), \exp^{-1}_x(z)
\>+O(\varepsilon^2)}{1+O(\varepsilon^2)}
=1-\<\nabla V(x), \exp^{-1}_{x}(z) \> + O(\varepsilon^2).
$$
Writing the $O(\varepsilon^2)$ term explicitly as $h(\cdot,x)$, we have 
$$
\begin{aligned}
1 = \int \f{d\bar{\nu}_x^\varepsilon}{d\bar{\mu}_x^\varepsilon}(z) d\bar{\mu}_x(z)
&= \int (1-\<\nabla V(x), \exp^{-1}_{x}(z) \>+h(z,x)) d\bar{\mu}_x(z)\\
&= 1 + \int h(z,x)d\bar{\mu}_x(z),
\end{aligned}
$$
since the $\nabla V$ term vanishes again, and this implies $\int h(z,x)d\bar{\mu}_x(z)=0$.

The expansion of the density $\f{d\nu_x^\varepsilon}{d\mu_x^\varepsilon}(z)$ is obtained the same way, starting from
$$
\f{d\nu_x^\varepsilon}{d\mu_x^\varepsilon}(z) = \f{e^{-V(z)}}{\int_{B_\varepsilon(x)} e^{-V(z)}d\mu_x^\varepsilon(z)}.
$$
\end{proof}

Similarly to \cref{lemma:density}, we have the density of the flat approximation of the non-uniform test measures with respect to the true non-uniform test measures:
\begin{lemma}
\label{lemma:density-2}
For every $x \in M$,
$$
\f{d\bar{\nu}_x^\varepsilon}{d\nu_x^\varepsilon}(z) = 1+h(z,x),
$$
where $h: M \times M \rightarrow \R$ is an a. e. smooth function such that $h(\cdot,x)=O(\varepsilon^2)$ for every $x \in M$ and $\int h(z,x)d\bar{\mu}_x(z)=\int h(z,x)d\bar{\nu}_x(z)=0$.
\end{lemma}

\begin{proof}
By the preceding lemma,
$$
\begin{aligned}
\f{d\bar{\nu}_x^\varepsilon}{d\nu_x^\varepsilon}(z) &= \f{d\bar{\nu}_x^\varepsilon}{d\bar{\mu}_{x}^\varepsilon}(z) \f{d\mu_x^\varepsilon}{d\nu_x^\varepsilon}(z) \\
&= \one_{B_{\varepsilon}(x)}(z) \f{1-\<\nabla V(x), \exp^{-1}_x(z) \> + h_1(z,x)}{1-\<\nabla V(x), \exp^{-1}_x(z) \> + h_2(z,x)} \\
&= \one_{B_{\varepsilon}(x)}(z) (1-\<\nabla V(x), \exp^{-1}_x(z) \> + h_1(z,x))\\
&\hspace{1.6cm} \times (1+\<\nabla V(x), \exp^{-1}_x(z) \> - h_2(z,x)+O(h_2(z,x)^2))\\
&= \one_{B_{\varepsilon}(x)}(z)  (1 + h_1(z,x) -h_2(z,x) +O(h_2(z,x)^2)) \\
&= \one_{B_{\varepsilon}(x)}(z) (1+h(z,x)),
\end{aligned}
$$
merging the remainder terms into $h(z,x)$ on the last line. The mean zero property of $h$ again follows because $\nu_x^\varepsilon$ and $\bar{\nu}_x^\varepsilon$ are probability measures.
\end{proof}

\subsection{An approximate transport map}
\label{approximate-transport}

Recall we wish to obtain an expansion of $W_1(\nu_{x_0}^\varepsilon, \nu_y^\varepsilon)$ in $\varepsilon$ and $\delta$. For this we propose an "approximate" transport map $T: B_\varepsilon(x_0) \rightarrow B_\varepsilon(y)$  which realizes the distance from \cref{variation-1} in the sense that $d(z,Tz) = L_1(\varepsilon,\delta)$. We first define such a map using a map between tangent spaces $\tilde{T}: \tilde{B}_\varepsilon(x_0) \rightarrow \tilde{B}_\varepsilon(y)$.

\begin{figure}
  \centering
\begin{tikzpicture}
\draw[rotate=30] (0,0) ellipse (1.4 and 0.5);
\draw (-1.25,-0.6) -- (-1.25,-3.5);
\draw (-1.25,-3.5) arc (180:360:1.25 and 0.5);
\draw [dashed] (-1.25,-3.5) arc (180:360:1.25 and -0.5);
\draw (1.25,-3.5) -- (1.25,0.6);  
\node[label=above left:{$\bar{\nu}_{x_0}^\varepsilon$}] at (-0.5,0) {};;
\node[circle, fill=black, scale=0.4, label=left:{$x_0$}] at (0,-3.5) {};
\draw [line width=0.4mm,->] (0,-3.5) -- (1.1,-2.7);
\node[label=left:{$\nabla V(x_0)$}] at (1.1,-2.7) {};;

\draw[rotate=15] (4.8,-1.35) ellipse (1.3 and 0.5);
\draw (3.75,-0.4) -- (3.75,-3.5);
\draw (3.75,-3.5) arc (180:360:1.25 and 0.5);
\draw [dashed] (3.75,-3.5) arc (180:360:1.25 and -0.5);
\draw (6.25,-3.5) -- (6.25,0.3);  
\node[label=above left:{$T_*\bar{\nu}_{x_0}^\varepsilon \sim \bar{\nu}_y^\varepsilon$}] at (4.3,0) {};
\node[circle, fill=black, scale=0.4, label=left:{$y$}] at (5,-3.5) {};
\draw [line width=0.4mm,->] (5,-3.5) -- (6,-3);
\node[label=above left:{$\nabla V(y)$}] at (6,-3) {};;

\draw [line width=0.4mm,->] (2,-2) -- (3,-2);
\node[label=above:{$T_*$}] at (2.5,-2) {};

\draw[dashed] (-3,-5) -- (8,-5) -- (8.5,-2.5) -- (6.5, -2.5);
\draw[dashed] (1.5, -2.5) -- (3.5, -2.5);
\draw[dashed] (-3,-5) -- (-2, -2.5) -- (-1.3, -2.5);
\node[label=above left:{$M$}] at (8,-5) {};;

\end{tikzpicture}
\caption{For small $\varepsilon>0$, the densities of the non-uniform measures $\bar{\nu}_{x_0}^\varepsilon, \bar{\nu}_y^\varepsilon$ resemble cylinders with a slant top, the slanting being described by the gradient of $V$. This motivates the choice of $T$ as a  scaled difference of the gradients of the tops, together with parallel translation from $B_\varepsilon(x_0)$ to $B_\varepsilon(y)$.}
\label{fig:transport-map}
\end{figure}
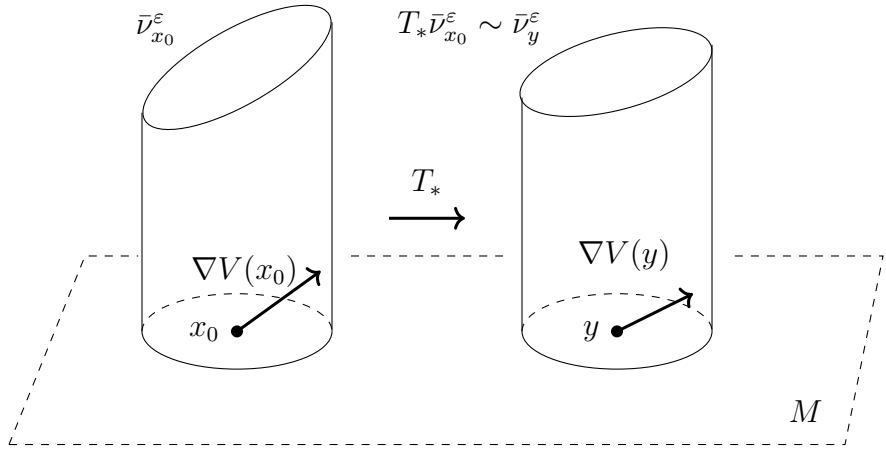

\begin{definition}
\label{T-tilde}
Define the map $\tilde{T}: \tilde{B}_\varepsilon(x_0) \rightarrow 
\tilde{B}_\varepsilon(y)$ as
\begin{equation}
\label{eq:transport}
\tilde{T}(w) := \ \newparallel_1 w- \f{1}{2}(\varepsilon^2 -
\|w\|^2) (\nabla V(y)- \newparallel_1 \nabla V(x_0)).
\end{equation}
Denote by $D \tilde{T}(w): T_{x_0}M \rightarrow T_y M$ its derivative at the point $w \in T_{x_0}M$. The notation $D\tilde{T}(w_1)(w_2)$ will mean the derivative at $w_1 \in T_{x_0}M$ in the direction of $w_2 \in T_{x_0}M$.
\end{definition}

\begin{remark}
  The first part of \eqref{eq:transport} is the same as in \cite{MR2484937} for the case of uniform measures. Parallel translation represents the closest analogue to translating all points by $y-x_0$ if the manifold were a Euclidean space, and turns out to be an approximate transport map for uniform test measures, i.e. $V=0$. In this case, the approximate transport is also optimal up to and including third order terms jointly in $\delta$ and $\varepsilon$.

  The second part in \eqref{eq:transport} is a novel adjustment to account for the non-uniformity of the measures being transported. It is motivated by \cref{lemma:density-mu-nu}, which says that for small $\varepsilon >0$, the densities of $\bar{\nu}_{x_0}^\varepsilon$ and $\bar{\nu}_y^\varepsilon$ are well approximated by affine functions on the respective supports. See \cref{fig:transport-map} for illustration. Any near-optimal transport of mass from $\bar{\nu}_{x_0}^\varepsilon$ to $\bar{\nu}_y^\varepsilon$ should thus consist of a translation from $B_\varepsilon(x_0)$ to $B_\varepsilon(y)$ together with a realignment of mass according to the difference of the gradients $\newparallel_1^{-1} \nabla V(y)-\nabla V(x_0)$, scaled with the distance from the centre of the support. This is because points on the boundary of $B_\varepsilon(x_0)$ should only be translated, while points near $x_0$ should be translated as well as moved in the direction of $-(\newparallel_1^{-1} \nabla V(y) - \nabla V(x_0))$. The above scaling by the factor $\f{1}{2}(\varepsilon^2-\|w\|^2)$ will yield a good enough approximate transport map because
  $$
  D \tilde{T}(w) = \mathrm{Id} + (\newparallel_1 \nabla V(y)-\nabla V(x_0)) \<w, \cdot\>,
  $$
  as shown in \cref{transport-map-jacobian}, which leads to a desirable cancellation of first order terms in the density $\f{d (T_*\bar{\nu}_{x_0}^\varepsilon)}{d \bar{\nu}^\varepsilon_y}$, shown in \cref{lemma:2}.

  In summary, \eqref{eq:transport} should be interpreted as parallel translation in the direction $v$ and distance $\delta$, with additional simultaneous shift in the direction $-(\newparallel_1^{-1} V(y)- \nabla V(x_0))$ and distance $\f{1}{2}(\varepsilon^2 -\|w\|^2)$ to adjust for non-uniformity of the test measures.
\end{remark}

To prove the next fact we will employ the auxilliary map $\hat{T}:\tilde{B}_\varepsilon(x_0) \rightarrow \tilde{B}_\varepsilon(x_0)$ defined as
\begin{equation}
\label{T-hat}
\hat{T}(w) :=  w- \f{1}{2}(\varepsilon^2 -
\|w\|^2) (\newparallel_1^{-1} \nabla V(y)-  \nabla V(x_0)),
\end{equation}
so that $\tilde{T} = \ \newparallel_1 \circ \; \hat{T}$.

\begin{lemma}
For $0<\varepsilon <1$ and $\delta >0$ sufficiently small,
the map $\tilde{T}$ is a well-defined diffeomorphism.
\end{lemma}
\begin{proof}
    Since $\newparallel_1: \tilde{B}_\varepsilon(x_0) \rightarrow \tilde{B}_\varepsilon(y)$ is a diffeomorphism, it is sufficient to show that $\hat{T}$ defined by \eqref{T-hat} is a diffeomorphism. We show the latter by finding the smooth inverse $\hat{T}^{-1}$. To further simplify notation we write
    $$
    \hat{T}(w) = w - \f{1}{2} (\varepsilon^2-\|w\|^2) \alpha \mathbf{e}
    $$
    with
    $$
    \mathbf{e} := \f{\newparallel_1^{-1} \nabla V(y)-  \nabla V(x_0)}{\|\newparallel_1^{-1} \nabla V(y)-  \nabla V(x_0))\|}, \quad \alpha := \|\newparallel_1^{-1} \nabla V(y)-  \nabla V(x_0))\| = O(\delta)>0.
    $$
    Decompose any vector $w \in \tilde{B}_\varepsilon(x_0)$ as $w = u+ r \mathbf{e}$ with $u \perp \mathbf{e}$ and $-\sqrt{\varepsilon^2-\|w\|^2} \leq r \leq \sqrt{\varepsilon^2-\|w\|^2}$ so that
    $$
    \hat{T}(u+r \mathbf{e}) = u +r \mathbf{e} -\f{1}{2}(\varepsilon^2-\|u\|^2-r^2)\alpha \mathbf{e} = u+h_u(r)\mathbf{e},
    $$
    where $h_u(r):= r- \f{1}{2}(\varepsilon^2-\|u\|^2-r^2)\alpha$. It is clear that
    $$
    h_u(\pm \sqrt{\varepsilon^2-\|u\|^2})=\pm\sqrt{\varepsilon^2-\|u\|^2}, \quad \f{dh_u}{dr}(r) = 1+\alpha r.
    $$
    Hence as long as $\alpha \varepsilon >-1$, the map
    $$
    h_u: [-\sqrt{\varepsilon^2-\|u\|^2},\sqrt{\varepsilon^2-\|u\|^2}] \rightarrow [-\sqrt{\varepsilon^2-\|u\|^2},\sqrt{\varepsilon^2-\|u\|^2}]
    $$
    is a diffeomorphism. Since the decomposition $w=u+r\mathbf{e}$ is unique and $(u,r) \mapsto h_u(r)$ is smooth, we conclude that $\hat{T}$ is a diffeomorphism. A fortiori, solving $h_u(r)=s$ for $r$ we may obtain the inverse explicitly as
    $$
    \hat{T}^{-1}(u+s\mathbf{e}) = u + \f{1}{\alpha}\left(-1+ \sqrt{1+\alpha^2(\varepsilon^2-\|u\|^2) +2\alpha y}\right) \mathbf{e},
    $$
    completing the proof.
\end{proof}

We refer back to \cref{T-tilde} for the definition of the symbols $\tilde{T}, D \tilde{T}(w)$.

\begin{lemma}
\label{transport-map-jacobian}
The Jacobian of $\tilde{T}$ satisfies for any $w \in \tilde{B}_\varepsilon(x_0)$
\begin{equation}
\label{jacobian-formula}
\det D \tilde{T}(w) = 1+\<\newparallel_1^{-1} \nabla V(y)-\nabla V(x_0), w\> + O(\delta^2 \varepsilon^2),
\end{equation}
where the $O(\delta^2 \varepsilon^2)$ term is uniformly bounded over $w$.
\end{lemma}
\begin{proof}
As $\newparallel_1: T_{x_0}M \rightarrow T_yM$ is a linear isomorphism, the derivative $D \newparallel_1 (w): T_{x_0}M \rightarrow T_yM$ at any $w \in T_{x_0}M$ coincides with parallel translation, i.e. $D \newparallel_1(w) = \; \newparallel_1$. Then by the chain rule
\begin{equation}
\label{T-tilde-derivative}
\begin{aligned}
D\tilde{T}(w) &= D (\newparallel_1 \circ \; \hat{T})(w) = \; \newparallel_1 \circ \; D\hat{T}(w) \\
&= \; \newparallel_1 \circ \Big( Id + \<w,\cdot\> (\newparallel_1^{-1} \nabla V(y) -\nabla V(x_0)) \Big) \\
&= \; \newparallel_1 + \<w, \cdot\>(\nabla V(y) - \newparallel_1 \nabla V(x_0))
\end{aligned}
\end{equation}
by taking the plain Euclidean derivative of \eqref{T-hat} to obtain the second line.

Consider $(e_i)_{i=1}^n$ an orthonormal basis of $T_{x_0}M$ and $(\newparallel_1 e_i)_{i=1}^n$ the corresponding parallel-translated orthonormal basis of $T_yM$.
With respect to these bases, the components of $D\tilde{T}(w)$ at any $w \in \tilde{B}_{\varepsilon}(x_0)$ are expressed using \eqref{T-tilde-derivative} as
\begin{equation}
\label{eq:determinant}
\begin{aligned}
\<D \tilde{T}(w)(e_i),\newparallel_1 e_j\> &=\<\newparallel_1 e_i, \newparallel_1 e_j\> +\<
w, e_i\> 
\<\nabla V(y)- \newparallel_1 \nabla V(x_0), \newparallel_1 e_j\> \\
&=\delta_{ij}+ \<w,e_i\>\<\newparallel_1^{-1}\nabla V(y)- \nabla V(x_0),e_j\>.
\end{aligned}
\end{equation}
Then
\begin{equation}
\label{eq:determinant-2}
\det D \tilde{T}(w) = \det \left(\delta_{ij} + \<w,e_i\> \<\newparallel^{-1}_1 \nabla V(y)-\nabla 
V(x_0),e_j\>\right)_{ij},
\end{equation}
where $\delta_{ij}$ stands for the Kronecker delta. Since
$$
\<w,e_i\> \<\newparallel^{-1}_1 \nabla V(y)-\nabla V(x_0),e_j\>=O(\delta\varepsilon),
$$
we can deduce that
\begin{equation}
  \label{det-DT}
\begin{aligned}
\det D\tilde{T}(w) &= 1 + \sum_{i=1}^n \<w,e_i\> \< \newparallel_1^{-1} \nabla V(y)- \nabla
V(x_0),e_i\>+O(\varepsilon^2 \delta^2)\\
&=1 + \<w, \newparallel_1^{-1} \nabla V(y) - \nabla V(x_0)\>+O(\varepsilon^2 \delta^2),
\end{aligned}
\end{equation}
as required.
Here we made use of the following well-known fact about determinants:
if $A(\varepsilon,\delta)= (a_{ij}(\varepsilon,\delta)) = \delta_{ij} + b_{ij}(\varepsilon,\delta)$ is a smooth $n
\times n$ matrix-valued function where $b_{ij}(\delta, \varepsilon)=O(\delta \varepsilon)$, then
$$
\det A(\varepsilon, \delta) = 1 + \sum_{i=1}^n b_{ii}(\varepsilon,\delta) + O(\delta^2 \varepsilon^2).
$$
\end{proof}

\begin{remark}
\label{transport-map-parallel-approx}
Since $\newparallel_1^{-1} \nabla V(y) - \nabla V(x_0)=O(\delta)$, it follows from the definition (\ref{eq:transport}) that
\begin{equation}
\label{eq:21}
\tilde{T}w = \; \newparallel_1 w + O(\delta \varepsilon^2) \quad \forall w \in
\tilde{B}_\varepsilon(x_0).
\end{equation}
Denoting $\tilde{w} = \tilde{T}w$, this implies also
\begin{equation}
\label{T-inverse-approximation}
\tilde{T}^{-1}\tilde{w} = \; \newparallel_1^{-1} \tilde{w} + O(\delta \varepsilon^2) \quad \forall \tilde{w} \in
\tilde{B}_\varepsilon(y),
\end{equation}
because the parallel translation is isometric.
As a consequence of (\ref{jacobian-formula}) and (\ref{T-inverse-approximation}), we thus have
\begin{equation}
\label{T-inverse-determinant}
\begin{aligned}
\det D\tilde{T}^{-1}(\tilde{w}) &= 1- \<\tilde{T}^{-1}\tilde{w},\newparallel_1^{-1} \nabla V(y)- \nabla V(x_0)\>+O(\varepsilon^2 \delta^2)\\
&= 1 - \< \tilde{w}, \nabla V(y)-\newparallel_1 \nabla V(x_0)\> + O(\delta^2 \varepsilon^2).
\end{aligned}
\end{equation}
\end{remark}

The exponential map is locally a diffeomorphism and $B_\varepsilon(x)=\exp_x (\tilde{B}_\varepsilon(x))$. This allows us to define the following diffeomorphism on the manifold, which we shall employ as our \textit{approximate} transport map between $\nu_{x_0}$ and $\nu_y$.

\begin{definition}
Define the map $T: B_\varepsilon(x_0) \rightarrow B_\varepsilon(y)$ as
\begin{equation}
  \label{4}
Tz := \exp_y (\tilde{T} \exp_{x_0}^{-1}(z)).
\end{equation}
\end{definition}

\begin{remark}
\label{transport-plan-explanation}
We explain the relationship between the map $T$ and the distance estimates of Section \ref{section:manifold-distance-estimates}. In the notation introduced in the context of geodesic variations $c_1, c_2$, if for any unit vector $w \in T_{x_0}M$ we label $z := \exp_{x_0}(\varepsilon w)$ then 
$$
Tz = \exp_y(\varepsilon \tilde{T} w) = \exp_y( \varepsilon \newparallel_1 w'),
$$
where $w'$ is given by (\ref{w-dash}). In the notation for geodesics of the variation (\ref{eta-theta-def}) we have $z=\eta(\varepsilon)$ and $Tz=\theta(\varepsilon)$, and thus in particular
$$
\begin{aligned}
d(z,Tz) &= d(\eta(\varepsilon), \theta(\varepsilon)) = L_1(\varepsilon, \delta),\\
d(z,p(z)) &= d(p(\eta(\varepsilon)), \eta(\varepsilon)) = L_2(\varepsilon, 0),\\
d(Tz, p(Tz)) &= d(p(\theta(\varepsilon)), \theta(\varepsilon)) = L_2(\varepsilon, \delta),
\end{aligned}
$$
where $p: M \rightarrow E$ is the projection map (\ref{projection}).
\end{remark}

While $T$ is not an exact transport map from $\nu_{x_0}^\varepsilon$ to $\nu_y^\varepsilon$, it turns out that the relevant Wasserstein distances may still be approximated using the pushforward measure $T_* \nu_{x_0}^\varepsilon$ in place of $\nu_y^\varepsilon$, and the distance $W_1(\nu_{x_0}^\varepsilon,T_*\nu_{x_0}^\varepsilon)$ can in turn be approximated by $W_1(\bar{\nu}_{x_0}^\varepsilon, T_*\bar{\nu}_{x_0}^\varepsilon)$. This is formalized below in \cref{wasserstein-approximations} as the proof requires preliminaries to follow. The merit of such approximation is that $W_1(\bar{\nu}_{x_0}^\varepsilon, T_*\bar{\nu}_{x_0}^\varepsilon)$ can be computed using the distance estimates obtained in the previous section. The upper bound for $W_1(\bar{\nu}_{x_0}^\varepsilon, T_*\bar{\nu}_{x_0}^\varepsilon)$ is established from the concrete transport plan $T$ and the lower bound is computed by using a concrete 1-Lipschitz function.

\begin{notation}
Denote $z := \exp_{x_0}(\varepsilon w)$ so that $Tz = \exp_y (\varepsilon \newparallel_1 w')$ where $w'$ is prescribed by (\ref{w-dash}), and let $r_0$ be the uniform injectivity radius on a fixed, compact neighbourhood of $x_0$. Recall that in the context of Proposition \ref{prop:variation-2} we defined the submanifold
$$
E := \{\exp_{x_0}(w): w \in T_{x_0}M, \la w, v \ra = 0\} \subset M,
$$
the projection $p: M \rightarrow E$, $p(x) := \mathrm{argmin}_{z \in E} d(x,z)$, $\nu \in \Gamma(TE^\perp)$ a unit normal vector field with $\nu(x_0)=v$ and the signed distance to projection, $f:B_{r_0/3}(x_0) \rightarrow \mathbb{R}$,
\begin{equation}
\label{distance-to-projection}
f(z) := \sgn (\langle \exp_{p(z)}^{-1} (z), \nu (z) \rangle) d(z,p(z)).
\end{equation}
Note that $f(Tz)=\tilde{L}_2(\varepsilon, \delta)$ and $f(z)=\tilde{L}_2(\varepsilon, 0)$ where $\tilde{L}_2$ was defined by (\ref{L2-tilde}), the signed length of the geodesic realizing the distance to projection.
\end{notation}

For notational convenience we label $z=\exp_{x_0}w$ for any $w \in T_{x_0}M$.

\begin{lemma}
\label{prop:lipschitz-function}
The signed distance to projection is 1-Lipschitz and for any vectors $v,w \in T_{x_0}M$ with $\|v\|=1, \|w\| \leq 1$ and $\delta, \varepsilon > 0$ sufficiently small,
\begin{equation}
\label{eq:14}
\begin{aligned}
f(Tz) - f(z) &= \delta\left(
1-\f{\varepsilon^2}{2}K_{x_0}(v,w)(\|w\|^2-\<v,w\>^2) -\f{\varepsilon^2}{2}\Hess_{x_0} V(v,v)(1-\|w\|^2)\right)\\
& \qquad + O(\delta^2 \varepsilon^2) + O(\delta \varepsilon^3),
\end{aligned}
\end{equation}
where $O(\delta^2\varepsilon^2)+O(\delta \varepsilon^3)$ is uniformly bounded in $z$.
\end{lemma}
\begin{proof}
We first show that $f$ is 1-Lipschitz. Since $z\mapsto \langle \exp_{p(z)}^{-1} (z), \nu (z) \rangle$ is continuous, $E \cap B_{r_0/3}(x_0)$ is the boundary between the two connected components 
\begin{equation}
\label{connected-components}
\{z \in B_{r_0/3}(x_0) : \langle \exp_{p(z)}^{-1} (z), \nu (z) \rangle > 0\}, \quad 
\{z \in B_{r_0/3}(x_0) : \langle \exp_{p(z)}^{-1} (z), \nu (z) \rangle < 0\}.
\end{equation}
First, consider $z_1, z_2 \in B_{r_0/3}(x_0)$ such that $\la \exp_{x_0}^{-1}(z_1),v\ra \la \exp_{x_0}^{-1}(z_2),v\ra \geq 0$. This means that $z_1, z_2$ lie in the same connected component of $B_{r_0/3}(x_0)$. Then by the triangle inequality,
$$
|f(z_1) - f(z_2)| = |d(z_1, p(z_1)) - d(z_2, p(z_2))| \leq d(z_1,z_2).
$$
Second, consider $z_1, z_2 \in B_{r_0/3}(x_0)$ 
such that 
$$
\la \exp_{x_0}(z_1),v\ra \la
\exp_{x_0}(z_2),v\ra < 0,
$$
meaning that $z_1, z_2$ lie in distinct components of $B_{r_0/3}(x_0)$. Suppose $\xi: [0,1] \rightarrow M$ is a length-minimizing geodesic connecting $z_1$ and $z_2$. Then 
\begin{equation}
\label{geodesic-containment}
\begin{aligned}
d(\xi(t),x_0) &\leq d(\xi(t), z_1) + d(z_1, x_0) \leq d(z_2, z_1) + d(z_1, x_0)\\
&\leq d(z_2, x_0) + 2 d(z_1, x_0) < \frac{1}{3} r_0 + \frac{2}{3}
r_0 = r_0.
\end{aligned}
\end{equation}
This means that the geodesic realising the distance between $z_1$ and $z_2$ lies in $B_{r_0}(x_0)$, and hence must pass through $E \cap B_{r_0}(x_0)$ because it is the boundary between the two connected components (\ref{connected-components}). Therefore, there exists $z_0 \in E \cap B_{r_0}(x_0)$ such that
$$
|f(z_1) - f(z_2)| = d(z_1, p(z_1)) + d(z_2,p(z_2))  
\leq d(z_1, z_0) + d(z_2, z_0) = d(z_1, z_2).
$$
To make the geodesic containment argument in \eqref{geodesic-containment}, we thus restrict $f$ to $B_{r_0 / 3}(x_0)$.
Finally, by \cref{difference-signed-lengths},
$$
\begin{aligned}
f(Tz)- f(z)
&= \tilde{L}_2(\varepsilon,\delta)-\tilde{L}_2(\varepsilon,0)\\
&= \delta\left( 1-\f{\varepsilon^2}{2}\left[K_{x_0}(v,w)(\|w\|^2-\<v,w\>^2)+ \Hess_{x_0}V(v,v)(1-\|w\|^2)\right]\right)\\
& \qquad + O(\delta^2 \varepsilon^2) +
O(\delta \varepsilon^3).
\end{aligned}
$$
\end{proof}

The signed distance to projection leads to the correct lower bound for $W_1(\bar{\nu}^\varepsilon_{x_0}, T_*\bar{\nu}^\varepsilon_{x_0})$ up to fourth order terms jointly in $\delta$ and $\varepsilon$:
\begin{lemma}
\label{lipschitz-function-wasserstein}
For $f$ defined by (\ref{distance-to-projection}) and all $v \in T_{x_0}M$ unit vectors and $\delta, \varepsilon > 0$ small enough,\begin{equation}
\label{eq:11}
\begin{aligned}
\int_{M} (f(Tz) - f(z))d\bar{\nu}_{x_0}^\varepsilon(z) &=
\delta\left(1-\f{\varepsilon^2}{2(n+2)}(\Ric_{x_0}(v,v)+2 \Hess_{x_0}V(v,v))\right) \\
& \qquad + O(\delta^2 \varepsilon^2) + O(\delta \varepsilon^3).
\end{aligned}
\end{equation}
\end{lemma}

\begin{proof}
Change the variable of integration to $w = \exp_{x_0}^{-1}(z)$ and recall $\bar{\mu}_{x_0}^\varepsilon := (\exp_{x_0})_* \tilde{\mu}_{x_0}^\varepsilon$. Then the integral on the left can be written using the density expression of \cref{lemma:density-mu-nu} and the estimate of \cref{prop:lipschitz-function} as
$$
\begin{aligned}
& \int_{M} (f(Tz) - f(z))\f{d\bar{\nu}_{x_0}^\varepsilon}{d\bar{\mu}_{x_0}^\varepsilon} (z) d\bar{\mu}_{x_0}^\varepsilon(z) \\
&= \int_M (f(Tz) - f(z)) (1-\<\nabla V(x), \exp^{-1}_x(z) \> + h(z,x_0)) d\bar{\mu}_{x_0}^\varepsilon(z) \\
&= \int_{T_{x_0}M}\delta \left( 1-\f{1}{2}K_{x_0}(v,w)(\|w\|^2-\<v,w\>^2)
-\f{1}{2} \Hess_{x_0}V(v,v) (\varepsilon^2-\|w\|^2)\right) \\
& \qquad \qquad \times \left(1-\<\nabla V(x_0), w \> + h(\exp_{x_0}w,x_0)\right) d\tilde{\mu}_{x_0}^\varepsilon(w) + O(\delta^2 \varepsilon^2) + O(\delta \varepsilon^3) \\
&= \int_{T_{x_0}M}\delta \left( 1-\f{1}{2}K_{x_0}(v,w)(\|w\|^2-\<v,w\>^2) -\f{1}{2} \Hess_{x_0}V(v,v) (\varepsilon^2-\|w\|^2)\right) d \tilde{\mu}_{x_0}^\varepsilon(w)\\
& \hspace{2cm} + O(\delta^2 \varepsilon^2) + O(\delta \varepsilon^3).
\end{aligned}
$$
On the last line, the linear term $\<\nabla V(x_0), w\>$ vanished by symmetry of the integration domain, and also we applied the mean zero property $\int_{T_{x_0}M} h(\exp_{x_0}w, x_0) d\tilde{\mu}_{x_0}^\varepsilon(w) = 0$. The products of these terms with the sectional curvature and Hessian terms were absorbed into the remainder $O(\delta \varepsilon^3)$.
Integrating the Hessian term in spherical coordinates, we have
$$
\begin{aligned}
\f{1}{2}\int_{T_{x_0}M} (\varepsilon^2-\|w\|^2)
d\tilde{\mu}_{x_0}^\varepsilon(w) &=
\f{\varepsilon^2}{2}-\f{1}{2|B_\varepsilon(x_0)|}\int_0^\varepsilon r^2 |\partial
B_r(x_0)|dr\\
&= \f{\varepsilon^2}{2} -\f{n\varepsilon^2}{2(n+2)} = \f{\varepsilon^2}{n+2}.
\end{aligned}
$$
Moreover, the sectional curvature term integrates to give the Ricci term as per \cref{lemma:ricci-average-ball}.
\end{proof}

\begin{theorem}
\label{wasserstein-pushforwards}
For any point $x_0 \in M$, vector $v \in T_{x_0}M$ with $\|v\| =1$ and $\delta, \varepsilon > 0$ sufficiently small,
\begin{equation}
W_1(\bar{\nu}_{x_0}^\varepsilon, T_*\bar{\nu}_{x_0}^\varepsilon) = \delta\left(1-\f{\varepsilon^2}{2(n+2)}
\left(\Ric_{x_0}(v,v)+2\Hess_{x_0}V(v,v)\right)\right) 
+ O(\delta^2 \varepsilon^2) + O(\delta \varepsilon^3).
\end{equation}
\end{theorem}

\begin{proof}
Using the transport map (\ref{eq:transport}) and applying \cref{variation-1} and \cref{lemma:density-mu-nu}, we have the upper bound
\begin{align}
\label{T-approximates-wasserstein}
&W_1(\bar{\nu}_{x_0}^\varepsilon,T_*\bar{\nu}_{x_0}^\varepsilon) \leq \int_{M}
d(Tz,z)d\bar{\nu}^\varepsilon_{x_0}(z) \nonumber\\
&=\int_{T_{x_0}M} \delta \left(1 -
\f{1}{2}K_{x_0}(v,w)(\|w\|^2-\<v, w\>^2)
-\f{1}{2}\Hess_{x_0}V(v,v)(\varepsilon^2-\|w\|^2) \right) \nonumber\\
& \qquad \times (1-\<\nabla V(x_0), w\> +h(\exp_{x_0}(w),x_0))
d\tilde{\mu}_{x_0}^\varepsilon(w) + O(\delta \varepsilon^3)+O(\delta^2 \varepsilon^2) \nonumber \\
&=\delta\left(1-\f{\varepsilon^2}{2(n+2)}
\left(\Ric_{x_0}(v,v)+2\Hess_{x_0}V(v,v)\right)\right)+O(\delta \varepsilon^3) + O(\delta^2 \varepsilon^2),
\end{align}
where we eliminated the $h$ and $\nabla V$ terms since
$$\int \<\nabla V(x_0), w\>
d\tilde{\mu}_{x_0}^\varepsilon(w)=0, \quad
\int h(\exp_{x_0}(w),x_0)d\tilde{\mu}_{x_0}^\varepsilon(w)=0.
$$
The Ricci curvature on the right appears by \cref{lemma:ricci-average-ball} and the Hessian term is obtained as in \cref{lipschitz-function-wasserstein} by integration in spherical coordinates.

For the converse direction, we use the Kantorovich-Rubinstein duality, choosing the 1-Lipschitz test function $f$ defined by (\ref{distance-to-projection}), so that by \cref{lipschitz-function-wasserstein} we have the lower bound 
$$
\begin{aligned}
W_1(\bar{\nu}_{x_0}^\varepsilon,T_*\bar{\nu}_{x_0}^\varepsilon) &\geq 
\int_{M} f(z) (dT_*\bar{\nu}_{x_0}^\varepsilon(z)-d\bar{\nu}_{x_0}^\varepsilon(z)) \\
&=\int_{B_\varepsilon(x_0)} (f(Tz)-f(z))d\bar{\nu}_{x_0}^\varepsilon(z)\\
&=\delta\left(1-\f{\varepsilon^2}{2(n+2)}(\Ric_{x_0}(v,v)+2 \Hess_{x_0}V(v,v))\right) \\
& \qquad + O(\delta^2 \varepsilon^2) + O(\delta \varepsilon^3).
\end{aligned}
$$
This shows the upper and lower bound agree up to terms $O(\delta \varepsilon^3)+O(\delta^2 \varepsilon^2)$ and their values are as required.
\end{proof}

We now prove the aforementioned approximation property and refer to \cref{notation:3} for the measures involved.
\begin{proposition}
\label{wasserstein-approximations}
For $x_0 \in M$ and $\delta, \varepsilon > 0$ sufficiently small,
$$
\begin{aligned}
W_1(\nu_{x_0}^\varepsilon,\nu_y^\varepsilon) &= 
W_1(\nu_{x_0}^\varepsilon, T_*\nu_{x_0}^\varepsilon)+O(\delta \varepsilon^3) + O(\delta^2 \varepsilon^2), \\
W_1(\nu_{x_0}^\varepsilon, T_*\nu_{x_0}^\varepsilon)&=W_1(\bar{\nu}_{x_0}^\varepsilon, T_*\bar{\nu}_{x_0}^\varepsilon)+
O(\delta \varepsilon^3) + O(\delta^2 \varepsilon^2),
\end{aligned}
$$
where $y= \exp_{x_0}(\delta v)$.
\end{proposition}
We need density estimates of the following two lemmas for the proof.
\begin{lemma}
\label{lemma:2}
It holds that
$$
\f{d(T_*\bar{\nu}_{x_0}^\varepsilon)}{d\bar{\nu}_y^\varepsilon}(z) = \mathbbm{1}_{B_y(\varepsilon)}(z) (1+h'(z,y)),
$$
where $h': M \times M \rightarrow \R$ is smooth a. e. and such that $h'(\cdot, y) = O(\delta \varepsilon^2) + O(\varepsilon^3)$ and $\int_{B_\varepsilon(y)} h'(z,y) d\bar{\nu}_y^{\varepsilon}(z)=0$ for all $y \in M$.
\end{lemma}
\begin{proof}
We first show the density estimate
\begin{equation}
\label{eq:1111}
\f{d(T_* \bar{\nu}_{x_0}^\varepsilon)}{d\bar{\mu}_y^\varepsilon}(z) = (1 -\< \nabla V(y), \exp_y^{-1}(z)\> + h(T^{-1}z, x_0)) \one_{B_\varepsilon(y)}(z).
\end{equation}
For any bounded Borel measurable $f: M \rightarrow \R$,
$$
\begin{aligned}
\int_M f(z) d(T_* \bar{\nu}_{x_0}^\varepsilon)&(z) = \int_M f(Tz) d\bar{\nu}_{x_0}^\varepsilon(z) \\
&= \int_M f(Tz) \f{d\bar{\nu}_{x_0}^\varepsilon}{d\bar{\mu}_{x_0}^\varepsilon}(z) d\bar{\mu}_{x_0}^\varepsilon(z)\\
&= \int_M f(Tz) (1-\<\nabla V(x_0), \exp_{x_0}^{-1}(z)\> +h(z,x_0)) d\bar{\mu}_{x_0}^\varepsilon (z)\\
&= \int_{T_{x_0}M} f(\exp_y \tilde{T}w)(1-\< \nabla V(x_0), w\>+h(\exp_{x_0}(w),x_0)) d\tilde{\mu}_{x_0}^\varepsilon(w),
\end{aligned}
$$
having applied \cref{lemma:density-mu-nu} on the second line. Substituting $\tilde{w}:= \tilde{T} w$ and using the change of variable formula, this becomes
$$
\begin{aligned}
\int_{T_yM} & f(\exp_y \tilde{w}) \left(1-\< \nabla V(x_0), \tilde{T}^{-1}\tilde{w} \> + h(\exp_{x_0} (\tilde{T}^{-1} \tilde{w}), x_0)\right) \left| \det D_{\tilde{w}} \tilde{T}^{-1} \right| d\tilde{\mu}_y^\varepsilon(\tilde{w})\\
&= \int_M f(z) \left(1- \< \nabla V(x_0), \tilde{T}^{-1} \exp^{-1}_y(z)\>+h(T^{-1}z,x_0)\right) \\
& \hspace{1.5cm} \times \left(1- \< \nabla V(y)-\newparallel_1 \nabla V(x_0), \exp_y^{-1}(z)\>+O(\delta^2 \varepsilon^2)\right) d\bar{\mu}_y^\varepsilon(z),
\end{aligned}
$$
having applied the determinant formula (\ref{T-inverse-determinant}). Referring back to \cref{transport-map-parallel-approx}, we know that
$$
\tilde{T}^{-1}\exp_y^{-1}(z) = \; \newparallel_1^{-1} \exp_y^{-1}(z) + O(\delta \varepsilon^2).
$$
Absorbing the $O(\delta \varepsilon^2)$ term into $h$, the integral above simplifies to
$$
\int_M f(z) \left(1-\< \nabla V(y), \exp_y^{-1}(z) \> + h(T^{-1}z, x_0)\right) d\bar{\mu}_y^\varepsilon(z),
$$
as required to obtain the density (\ref{eq:1111}).

Finally, using \cref{lemma:density-mu-nu} and (\ref{eq:1111}) we obtain
\begin{equation}
\label{eq:20}
\begin{aligned}
&\f{dT_*\bar{\nu}_{x_0}^\varepsilon}{d\bar{\nu}_y^\varepsilon}(z)\\
&=\f{d\bar{\mu}_{y}^\varepsilon}{d\bar{\nu}_y^\varepsilon}(z) 
\f{dT_*\bar{\nu}_{x_0}^\varepsilon}{d\bar{\mu}_{y}^\varepsilon}(z) 
=\lc\f{d\bar{\nu}_y^\varepsilon}{d\bar{\mu}_{y}^\varepsilon}(z)\rc^{-1}
\f{dT_*\bar{\nu}_{x_0}^\varepsilon}{d\bar{\mu}_{y}^\varepsilon}(z) \\ 
&=\mathbbm{1}_{B_y(\varepsilon)}(z)(1+\<\nabla V(y), \exp_y^{-1}(z)\>-h(z,y) + \<\nabla V(y), \exp_y^{-1}(z)\>^2 + O(\varepsilon^3)) \\
& \hspace{1.6cm} \times (1-\<\nabla V(y), \exp_y^{-1}(z)\>+h(T^{-1}z,x_0)) \\
&=\mathbbm{1}_{B_y(\varepsilon)}(z) \bigg(h(T^{-1}z,x_0)-h(z,y) \\
&\hspace{2.5cm} + \< \nabla V(y), \exp_y^{-1}(z) \>(h(T^{-1}z, x_0) +h(z,y))+ O(\varepsilon^3)\bigg).
\end{aligned}
\end{equation}
The function $(y,z) \mapsto h(T^{-1}z, x_0)-h(z,y)$ is smooth a.e. on $M \times M$ and $z \mapsto h(T^{-1} z, x_0)-h(z,y)$ is of order $\varepsilon^2$ for every $y$ and vanishing for $y=x_0$ (i.e. $\delta=0$). Hence $h(T^{-1}z, x_0)-h(z,y)=O(\min(\delta, \varepsilon^2))$ and by smoothness of $h$ thus of order $\delta \varepsilon^2$, see \cref{smooth-joint-order}. Moreover, $\exp_y^{-1}(z) = O(\varepsilon)$ and $h(z,y)+h(T^{-1}z,x_0+h(z,y))=O(\varepsilon^2)$ imply that the term on the last line is $O(\varepsilon^3)$.
\end{proof}

\begin{remark}
The cancellation in (\ref{eq:20}) occurs because of the specific choice of $T$, and is essential for the $O(\delta \varepsilon^2)$ estimate.
\end{remark}

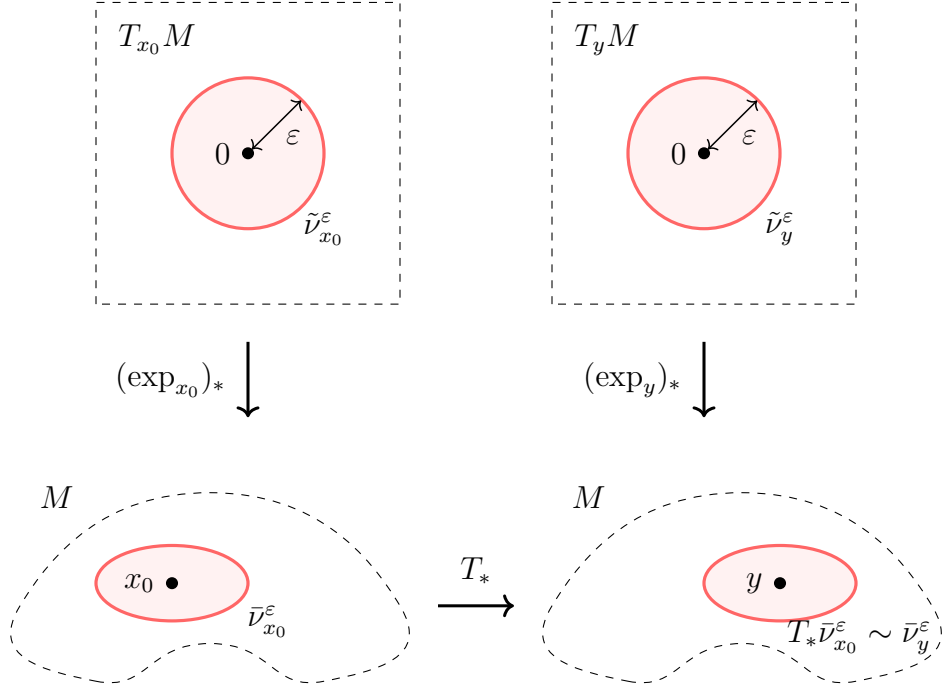
\begin{figure}
\centering
\begin{tikzpicture}
\draw[dashed] (0,0) rectangle (4,4);
\filldraw[color=red!60, fill=red!5, very thick] (2,2) circle (1);
\node[circle, fill=black, scale=0.4, label=left:{$0$}] at (2,2) {};
\node[label=below right:{$T_{x_0}M$}] at (0,4) {};
\draw [line width=0.4mm,->] (2,-0.5) -- (2,-1.5);
\node[label=left:{$(\exp_{x_0})_*$}] at (2,-1) {};
\node[label=center:{$\tilde{\nu}_{x_0}^\varepsilon$}] at (3,1) {};
\draw [line width=0.2mm,<->] (2.05,2.05) -- (2.7,2.7);
\node[label=below:{$\varepsilon$}] at (2.6,2.6) {};

\draw[dashed] (6,0) rectangle (10,4);
\filldraw[color=red!60, fill=red!5, very thick] (8,2) circle (1);
\node[circle, fill=black, scale=0.4, label=left:{$0$}] at (8,2) {};
\node[label=below right:{$T_{y}M$}] at (6,4) {};
\draw [line width=0.4mm,->] (8,-0.5) -- (8,-1.5);
\node[label=left:{$(\exp_{y})_*$}] at (8,-1) {};
\node[label=center:{$\tilde{\nu}_y^\varepsilon$}] at (9,1) {};
\draw [line width=0.2mm,<->] (8.05,2.05) -- (8.7,2.7);
\node[label=below:{$\varepsilon$}] at (8.6,2.6) {};

\node[label=above left:{$M$}] at (0,-3) {};
\draw [dashed] plot [smooth cycle, blue, tension=1] coordinates {
(-1,-4) (0,-5) (1.5, -4.5) (3, -5) (4,-4) (1.5,-2.5)};
\filldraw[color=red!60, fill=red!5, very thick] (1,-3.7) ellipse (1 and 0.5);
\node[circle, fill=black, scale=0.4, label=left:{$x_0$}] at (1,-3.7) {};
\node[label=below right:{$\bar{\nu}_{x_0}^\varepsilon$}] at (1.7,-3.7) {};

\draw [line width=0.4mm,->] (4.5,-4) -- (5.5,-4);
\node[label=above:{$T_*$}] at (5,-4) {};

\node[label=above left:{$M$}] at (7,-3) {};
\draw [dashed] plot [smooth cycle, blue, tension=1] coordinates {
(6,-4) (7,-5) (8.5, -4.5) (10, -5) (11,-4) (8.5,-2.5)};
\filldraw[color=red!60, fill=red!5, very thick] (9,-3.7) ellipse (1 and 0.5);
\node[circle, fill=black, scale=0.4, label=left:{$y$}] at (9,-3.7) {};
\node[label=below right:{$T_* \bar{\nu}_{x_0}^\varepsilon \sim \bar{\nu}_y^\varepsilon$}] at (8.8,-3.9) {};

\end{tikzpicture}
\caption{Illustration of the relationship between $T_* \bar{\nu}_{x_0}^\varepsilon$ and $\bar{\nu}_y^\varepsilon$. The two measures are equivalent with mutual density of the form $1+O(\delta \varepsilon^2)$.
}
\label{fig:4}
\end{figure}

The estimate from \cref{lemma:2} carries over to $\nu_{x_0}$ and $\nu_y$ as we shall prove below.
\begin{lemma}
It holds that
$$
\f{d(T_*\nu_{x_0}^\varepsilon)}{d\nu_y^\varepsilon}(z) = \mathbbm{1}_{B_y(\varepsilon)}(z) (1+h'(z,y)),
$$
where $h': M \times M \rightarrow \R$ is smooth \mbox{a.e.} and such that $h'(\cdot, y) = O(\delta \varepsilon^2)$ and 
$$
\int_{B_\varepsilon(y)} h'(z,y) d\bar{\nu}_y^{\varepsilon}(z)=0 \quad \forall y \in M.
$$
\end{lemma}
\begin{proof}
By \cref{lemma:density-2} we may write 
$$
d\nu^\varepsilon_{x_0}(z) = (1+h(x_0,z))d\bar{\nu}_{x_0}^\varepsilon(z), \quad d\nu^\varepsilon_{y}(z) = (1+h(y,z))d\bar{\nu}_{y}^\varepsilon(z)
$$
for a function $h: M \times M \rightarrow \R$ smooth \mbox{a. e.} and such that
$h(\cdot,x)=O(\varepsilon^2)$ and $\int h(z,x)d\bar{\mu}_x(z)=0$ for every $x \in M$.
Then
$$
\begin{aligned}
d(T_* \nu_{x_0})(z) &= T_*[(1+h(x_0,z))d\bar{\nu}_{x_0}^\varepsilon(z)] \\
&=(1+h(x_0,T^{-1}z))d(T_* \bar{\nu}_{x_0}^\varepsilon)(z) \\
&=(1+h(y,z)+O(\delta \varepsilon^2)) d(T_* \bar{\nu}_{x_0}^\varepsilon)(z)\\
&=(1+h(y,z)+O(\delta \varepsilon^2)) \f{d(T_* \bar{\nu}_{x_0}^\varepsilon)}{d\bar{\nu}_y^\varepsilon}(z) d\bar{\nu}_y^\varepsilon(z)\\
&=(1+h(y,z)+O(\delta \varepsilon^2)) (1+O(\delta \varepsilon^2)) d\bar{\nu}_y^\varepsilon(z)\\
&=(1+h(y,z)+O(\delta \varepsilon^2)) d\bar{\nu}_y^\varepsilon(z)\\
&=(1+O(\delta \varepsilon^2))d\nu_y^\varepsilon(z).
\end{aligned}
$$
On the third line we applied that $h(y,z)-h(x_0, T^{-1}z) = O(\delta \varepsilon^2)$ by the same argument as in the proof of \cref{lemma:2}. The mean zero property of $h'(\cdot,y)$ again follows from $1+h(\cdot,y)$ being a density with respect to a probability measure.
\end{proof}

\begin{proof}[Proof of \cref{wasserstein-approximations}]
For the first equality, $\forall f \in \mathrm{Lip}_1(M)$:
$$
\begin{aligned}
    \int f(z) (d(T_* \nu_{x_0}^\varepsilon)(z)-d\nu_{x_0}^\varepsilon(z)) &= \int f(z) \left(\f{d(T_* \nu_{x_0}^\varepsilon)}{d\nu_y^\varepsilon}(z)d\nu_y^\varepsilon(z)-d\nu_{x_0}^\varepsilon(z) \right)\\
    &= \int f(z) ((1+h'(y,z))d\nu_y^\varepsilon(z) - d\nu_{x_0}^\varepsilon(z)) \\
    &= \int f(z) (d\nu_y^\varepsilon(z)-d\nu_{x_0}^\varepsilon(z)) \\
    & \quad + \int (f(z)-f(y)) h'(y,z) d\nu_y^\varepsilon(z)
\end{aligned}
$$
and the last term is of order $\delta \varepsilon^3$ as $h'=O(\delta \varepsilon^2)$ by the preceding lemma.

For the second equality, we know from \cref{wasserstein-pushforwards} and the equality (\ref{T-approximates-wasserstein}) that $T$ satisfies
$$
\int d(z,Tz) d\bar{\nu}_{x_0}^\varepsilon(z) = W_1(\bar{\nu}_{x_0}^\varepsilon, T_*\bar{\nu}_{x_0}^\varepsilon) + O(\delta^2\varepsilon^2)+O(\delta \varepsilon^3).
$$
By \cref{lipschitz-function-wasserstein}, there exists $f \in \mathrm{Lip}_1(M)$ such that
\begin{equation}
\label{concrete-f}
\int f(z)(d(T_*\bar{\nu}_{x_0})(z) - d\bar{\nu}_{x_0}(z)) = W_1(\bar{\nu}_{x_0}^\varepsilon, T_*\bar{\nu}_{x_0}^\varepsilon) + O(\delta^2\varepsilon^2)+O(\delta \varepsilon^3).
\end{equation}
\cref{variation-1} gives $d(Tz,z)= d(y,x_0)(1+O(\varepsilon^2))$, hence we obtain the upper bound
$$
\begin{aligned}
W_1(\nu_{x_0}^\varepsilon, T_* \nu_{x_0}^\varepsilon) &\leq \int d(Tz,z) d\nu_{x_0}^\varepsilon(z) \\
&= \int d(Tz,z)(1+h(x_0,z))d\bar{\nu}_{x_0}^\varepsilon(z)\\
&= \int d(Tz,z) d\bar{\nu}_{x_0}^\varepsilon(z) + \int d(y,x_0)(1+O(\varepsilon^2))h(x_0,z) d\bar{\nu}_{x_0}^\varepsilon(z)\\
&= W_1(\bar{\nu}_{x_0}, T_* \bar{\nu}_{x_0})+ O(\delta^2 \varepsilon^2)+O(\delta \varepsilon^3),
\end{aligned}
$$
having applied the property $\int h(x_0,z) d\bar{\nu}_{x_0}^\varepsilon(z)=0$ on the last line.

\noindent By \cref{prop:variation-2}, for $\delta,\varepsilon$ sufficiently small the $f$ from \cref{lipschitz-function-wasserstein} satisfies
$$
\begin{aligned}
f(Tz)-f(z)&=\tilde{L}_2(\varepsilon, \delta)-\tilde{L}_2(\varepsilon,0) \\
&=L_1(\varepsilon,\delta)+O(\delta^2\varepsilon^2) +O(\delta \varepsilon^3) \\
&=d(Tz,z)+O(\delta^2\varepsilon^2) +O(\delta \varepsilon^3) \\
&= d(x_0,y)(1+O(\varepsilon^2)) +O(\delta^2\varepsilon^2)+O(\delta \varepsilon^3),
\end{aligned}
$$
which leads to the lower bound
$$
\begin{aligned}
W_1(\nu_{x_0}^\varepsilon, T_* \nu_{x_0}^\varepsilon) &\geq 
\int f(z)(d(T_* \nu_{x_0}^\varepsilon)(z)-d\nu_{x_0}^\varepsilon(z)) \\
&= \int (f(Tz) -f(z)) d\nu_{x_0}^\varepsilon(z) \\
&= \int (f(Tz) -f(z)) (1+h(x_0,z)) d\bar{\nu}_{x_0}^\varepsilon(z) \\
&=\int (f(Tz)-f(z)) d\bar{\nu}_{x_0}^\varepsilon(z) + \int d(y,x_0)(1+O(\varepsilon^2))h(x_0,z)d\bar{\nu}_{x_0}^\varepsilon(z)\\
&=W_1(\bar{\nu}_{x_0}, T_*\bar{\nu}_{x_0})+ O(\delta^2 \varepsilon^2)+O(\delta \varepsilon^3),
\end{aligned}
$$
having applied the mean zero property of $h$ again on the last line.
\end{proof}

\cref{wasserstein-pushforwards} and \cref{wasserstein-approximations} together yield \cref{thm:0}.

\section{Application to random geometric graphs}
\label{rgg-application}
Hoorn et al. \cite{hoorn-2023} showed that the coarse curvature of the random geometric graph sampled from a uniform Poisson point process on a Riemannian manifold converges in expectation to the smooth Ricci curvature as the intensity of the process increases.

\cref{thm:0} allows us to extend their result to manifolds equipped with a smooth potential $V: M \rightarrow \R$, sampling now from a Poisson point process with increasing intensity, proportional to the non-uniform measure $e^{-V(z)} \vol(dz)$. By sampling we mean using the empirical measures of the Poisson point process as the vertices of the random geometric graph. Our novel method allows to deal with this non-uniformity as well as the non-uniformity incurred by the exponential mapping. Since the main result of \cref{graph-continuum-limit-curvature} is at any rate local, we will assume that $M$ is compact.

We recall a few definitions in order to state our extended result. We refer to \cite{MR3791470} and \cite{MR1986198} for a background on Poisson point processes and random geometric graphs, respectively.
\begin{definition}
\label{poisson-point-process}
Let $(\mathcal{X}, \mathcal{B}, \mu)$ be a measure space with $\mu(\mathcal{X}) <\infty$, $\mathcal{M}(\mathcal{X})$ the set of measures on $\mathcal{X}$ and $(\Omega, \mathcal{F}, \mathbb{P})$ a probability space. A random measure $\mathcal{P}: \Omega \times \mathcal{B} \rightarrow [0,\infty]$, or equivalently $\mathcal{P}: \Omega \rightarrow \mathcal{M}(\mathcal{X})$, is said to be a Poisson point process on $\mathcal{X}$ with intensity $\mu$  if
\begin{enumerate}
    \item $\forall A \in \mathcal{B}$ with $\mu(A)<\infty$, $\mathcal{P}(\cdot,A)$ is a random variable with $\mathrm{Poisson}(\mu(A))$ distribution,
    \item $\forall A,B \in \mathcal{B}$ disjoint and $\mu$-finite, the random variables $\mathcal{P}(\cdot,A)$ and $\mathcal{P}(\cdot, B)$ are independent,
    \item $\forall \omega \in \Omega:$ $\mathcal{P}(\omega, \cdot)$ is a measure on $(\mathcal{X}, \mathcal{M})$.
\end{enumerate}
\end{definition}

Given a measure $\mu$, the Poisson point process can be constructed as follows. Let $N_0$ be a $\mathrm{Poisson}(\mu(\mathcal{X}))$  random variable, let $(X_i)_{i \in \mathbb{N}}$ be independent $X$-valued random variables which are also independent of $N_0$ and distributed as $\f{1}{\mu(\mathcal{X})}\mu$. Define the random measures
\begin{equation}
\label{poisson-process-deltas}
\mathcal{P}: \Omega \rightarrow \mathcal{M}(\mathcal{X}), \quad
\omega  \mapsto \sum_{i=1}^{N_0(\omega)} \delta_{X_i(\omega)}.
\end{equation}
$\mathcal{P}$ satisfies the properties in \cref{poisson-point-process}, and it can be shown that all Poisson point processes with a finite intensity measure take this form, see \cite[Chap. 6]{MR3791470}. 

The following transformation property of Poisson processes will be needed later:
\begin{lemma}
\label{ppp-pushforward}
    \cite{MR3791470} Let $(\mathcal{Y}, \mathcal{B}')$ be another measurable space and $\psi: \mathcal{X} \rightarrow \mathcal{Y}$ a measurable map. Then the push-forward process $\mathcal{P}': \Omega \times \mathcal{B}' \rightarrow  [0,\infty]$ defined by $\mathcal{P}'(\omega, \cdot) := \psi_* \mathcal{P}(\omega, \cdot)$ is a Poisson point process on $\mathcal{Y}$ with intensity $\psi_* \mu$. Writing the process $\mathcal{P}$ in the form (\ref{poisson-point-process}), the push-forward process can be written as
    $$
    \mathcal{P}'(\omega,\cdot) = \sum_{i=1}^{N_0(\omega)} \delta_{\psi(X_i(\omega))}.
    $$
\end{lemma}
\begin{proof}
    The random variable $\mathcal{P}'(\cdot,A) = \mathcal{P}(\cdot,\psi^{-1}(A))$ has a $\mathrm{Poisson}(\mu(\psi^{-1}(A)))$ distribution, meaning that $\mathcal{P}'$ has intensity $\psi_*\mu$. Properties (ii) and (iii) are preserved by the push-forward.
\end{proof}

Let $\mathcal{P}$ be a Poisson point process on a fixed neighbourhood of finite Lebesgue measure on the complete, orientable Riemannian manifold $M$ with Riemannian distance $d$ and $x_0, y \in M$ two points in the fixed neighbourhood. We implicitly identify every Dirac measure $\delta_x$ from the Poisson point process with the point $x$.
\begin{definition} 
\label{random-geometric-graph}
A random geometric graph sampled from $\mathcal{P}$ with roots $x_0, y$ and connectivity radius $\varepsilon >0$ is the weighted graph denoted by $G(x_0,y,\varepsilon)$ with nodes given by
$$
\mathcal{V}(\omega) = \{X_{i}(\omega) : 1\leq i \leq N_0(\omega)\} \cup \{x_0, y\},
$$
where the variables $X_i$ originate from the Poisson process $\mathcal{P}$ by (\ref{poisson-process-deltas}), and edges
$$
\{ (u,v) : u,v \in \mathcal{V}(\omega), d(u,v) < \varepsilon \}.
$$
The edge weights are given by manifold distance $d(u,v)$ for every edge $(u,v)$. Denote the graph distance by
$$
d_G(x,z) := \inf \left\{\sum_{i=1}^m d(x_i,x_{i+1}) : x_1=x, x_m=z, (x_i,x_{i+1}) \; \mathrm{ edge}, m \in \mathbb{N}\right\}.
$$
\end{definition}

\begin{notation}
\label{graphs-notation}
Let $V:M \rightarrow \mathbb{R}$ be a smooth potential on $M$, fix $x_0 \in M$ and denote by $(\mathcal{P}_n)_{n \in \mathbb{N}}$ the sequence of Poisson point processes with intensity measures 
\begin{equation}
\label{ppp-intensities}
n e^{-V(z)+V(x_0)}d\vol(z).
\end{equation}
For sequences $(\delta_n)_{n \in \mathbb{N}}, (\varepsilon_n)_{n \in \mathbb{N}}$ with
$\delta_n,\varepsilon_n \rightarrow 0$ as $n \rightarrow \infty$, label the sequence of
points $y_n := \exp_{x_0}(\delta_n v)$ approaching $x_0$ from a fixed direction $v \in
T_{x_0}M$. This gives rise to a sequence of rooted random graphs
$G_n(x_0,y_n,\varepsilon_n)$ according to \cref{random-geometric-graph} and the empirical
measure  representation by (\ref{poisson-process-deltas}) for the Poisson processes,
$\mathcal{P}_n(\omega, \cdot)=\sum_{i=1}^{N_{0,n}(\omega)} \delta_{X^{(n)}_i(\omega)}$.
For each $n\in \mathbb{N}$ denote the set of nodes
$$
\mathcal{V}_n(\omega) := \{X_{i}^{(n)}(\omega) : 1 \leq i \leq N_{0,n}(\omega)\} \cup \{x_0, y\},
$$
where the variables $\{X_i^{(n)}, N_{0,n}: i,n \in \mathbb{N}\}$ are in addition assumed to be independent. For the graph distance $d_{G_n}$ define the $\delta_n$-neighbourhood of $x$ in $G_n$ as
$$
B^{G_n}_{\delta_n}(x) := \{z \in \mathcal{V}_n : d_{G_n}(x,z) < \delta_n\}.
$$
For any node $x \in \mathcal{V}_n$ denote the normalized empirical measures
\begin{equation}
\label{empirical-measure}
\eta^{\delta_n}_x(\{z\}) := \begin{cases}
    \f{1}{\# (B^{G_n}_{\delta_n}(x))}, & z \in B^{G_n}_{\delta_n}(x),\\
    0, & \mathrm{otherwise},
\end{cases}
\end{equation}
which is equivalently written as
$$
\eta^{\delta_n}_x(A):= \sum_{z\in  B^{G_n}_{\delta_n}(x)} \f{\delta_z(A)}{\# (B^{G_n}_{\delta_n}(x))}.
$$
Denote the graph curvature of $G_n(x_0, y_n,\varepsilon_n)$ at $x_0$ in the direction of $v=\delta_n^{-1}\exp_{x_0}^{-1}(y_n)$ as
\begin{equation}
\label{coarse-curvature-graph}
\kappa_n(x_0, y_n) := 1- \f{W^{G_n}_1(\eta^{\delta_n}_{x_0}, \eta^{\delta_n}_{y_n})}{d_{G_n}(x_0,y_n)}.
\end{equation}
\end{notation}

\begin{remark}
For the sequence of graphs $G_n(x_0,y_n, \varepsilon_n)$ we thus have the corresponding collections of random empirical measures $(\eta^{\delta_n}_x)_{x \in \mathcal{V}_n}$ and coarse curvatures $\kappa_n(x_0,y_n)$ at $x_0$, the limit of which we are looking to establish in the sequel.
We emphasise that these are all random objects, i.e. dependent on $\omega \in \Omega$, although we suppress this from the notation.
\end{remark}

Convergence of the coarse graph curvature to the generalized Ricci curvature can be proved under an assumption on the rate of convergence of the graph parameters $\delta_n, \varepsilon_n$ to zero. For two sequences $(a_n), (b_n)$ we denote $a_n \sim b_n$ as $n \rightarrow \infty$ if there exist $c,C>0$ such that $cb_n \leq a_n \leq Cb_n$ for all $n \in \mathbb{N}$.
\begin{assumption}
\label{delta-epsilon-assumption}
Denoting $N=\dim M$, assume that $\varepsilon_n \leq \delta_n$ and
$$
\varepsilon_n \sim (\log n)^a n^{-\alpha}, \quad
\delta_n \sim (\log n)^b n^{-\beta},
$$
where the constants $\alpha,\beta, a,b$ are such that 
$$
0 < \beta \leq \alpha, \quad \alpha + 2\beta \leq \f{1}{N},
$$
and in case of equality in either of the above conditions additionally 
$$
\begin{cases}
    a \leq b & \mathrm{if } \; \alpha = \beta,\\
    a+2b > \f{2}{N} & \mathrm{if } \; \alpha+2\beta=\f{1}{N}.
\end{cases}
$$
\end{assumption}

Below is the main result of the section, with the relevant objects set up in \cref{graphs-notation}.
\begin{theorem}
\label{graph-continuum-limit-curvature}
Let $\kappa_n(x_0,y)$ be the graph curvatures corresponding to the rooted random graphs $G_n(x_0,y_n, \varepsilon_n)$ with $y_n=\exp_{x_0}(\delta_n v)$ generated by the sequence of Poisson processes with intensity measures 
$$
n e^{-V(z)+V(x_0)}\vol(dz),
$$
with nodes and edges as specified in \cref{graphs-notation}.
Under \cref{delta-epsilon-assumption}, there exists $(c_n)_{n \in \mathbb{N}}$ such that $\lim_{n\rightarrow \infty} c_n=0$ and
\begin{equation}
\label{G-M-distance}
\E\left[\left| W_1^{G_n}(\eta_{x_0}^{\delta_n}, \eta_{y_n}^{\delta_n})-W_1(\nu_{x_0}^{\delta_n}, \nu_{y_n}^{\delta_n}) \right|\right] \leq c_n\delta_n^3,
\end{equation}
which implies
\begin{equation}
\label{graph-curvature-convergence}
\lim_{n \rightarrow \infty} \E\left[\left| 
\f{2(N+2)}{\delta_n^2} \kappa_n - (\Ric_{x_0}(v,v)+2 \Hess_{x_0} V(v,v))
\right| \right] = 0.
\end{equation}
\end{theorem}

\begin{remark}
    By contrast, the theorem of Hoorn et al. \cite[Theorem 3]{hoorn-2023} used a Poisson point process of increasing uniform volume intensity and concluded pointwise convergence in expectation as in \eqref{graph-curvature-convergence} with $V=0$. \cref{graph-continuum-limit-curvature} above thus provides an extension for arbitrary $V: M \rightarrow \R$.
\end{remark}

\begin{remark}
\label{1000}
The asymptotic bound (\ref{G-M-distance}) is indeed sufficient for (\ref{graph-curvature-convergence}), because in combination with \cref{thm:0} and by the triangle inequality, we have the upper bound for the latter given by
$$
\begin{aligned}
&\f{2(N+2)}{\delta_n^2}\E\left[\left|\kappa_n(x_0,y_n)-\left(1-\f{W_1(\nu_{x_0}^{\delta_n}, \nu_{y_n}^{\delta_n})}{\delta_n}\right)\right|\right]\\
&\quad +\left|\f{2(N+2)}{\delta_n^2}\left(1-\f{W_1(\nu_{x_0}^{\delta_n}, \nu_{y_n}^{\delta_n})}{\delta_n}\right)- (\Ric_{x_0}(v)+2 \Hess_{x_0} V(v))\right|.
\end{aligned}
$$
The second term converges to 0 by (\ref{coarse-ricci}) and the first term can be written as
$$
\f{2(N+2)}{\delta_n^3} \E\left[\left| W_1^{G_n}(\eta_{x_0}^{\delta_n}, \eta_{y_n}^{\delta_n})-W_1(\nu_{x_0}^{\delta_n}, \nu_{y_n}^{\delta_n}) \right|\right],
$$
which vanishes as $n\rightarrow \infty$ if (\ref{G-M-distance}) holds. 
\end{remark}

For the rest of this section, we focus on establishing (\ref{G-M-distance}). We follow the methods laid out in \cite{hoorn-2023}, which consist in showing that the graph distances $d_{G_n}$ can be extended to the manifold to give a close approximation of the Riemannian distance $d$, and then using such extension to estimate the difference (\ref{G-M-distance}).

Let $(\lambda_n)_{n \in \mathbb{N}}$ be the sequence given by
\begin{equation}
    \label{lambda-sequence}
\lambda_n := (\log n)^\frac{2}{N} n^{-\frac{1}{N}}.
\end{equation}
This can be regarded in view of the following definition  as a "distance extension radius". Denote by $B_{\lambda_n}(x)$ the geodesic ball in the original manifold $M$ centered at point $x$ with radius $\lambda$.

\begin{definition}
    \label{extended-metric}
Let $G_n=G_n(x_0,y_n,\delta_n)$ be the rooted random geometric graphs of \cref{graph-continuum-limit-curvature} and for any $x,y \in M$ denote by $G_n \cup \{x,y\}$ the graph extended by nodes $\{x,y\}$ and extended by edges $\{(x,z) : z\in B_{\lambda_n}(x) \cap G_n\}\cup \{(y,z) : z\in B_{\lambda_n}(y) \cap G_n\}$ with extension radius $\lambda_n$. Further, let $(d_n)_{n \in \mathbb{N}}$ be the sequence of random functions $d_n: M \times M \rightarrow \R$ defined by
$$
d_n(x,y) := d(x,\tilde{x}) + d(y,\tilde{y}) + d_{G_n}(\tilde{x}, \tilde{y}),
$$
where
$$
\tilde{x} = \underset{z \in B_{\lambda_n}(x) \cap G_n}{\mathrm{argmin}} d(x,z), \quad \tilde{y} = \underset{z \in B_{\lambda_n}(y) \cap G_n}{\mathrm{argmin}} d(y,z),
$$
if both $\tilde{x}, \tilde{y}$ exist, and set $d_n(x,y)=\infty$ otherwise.
\end{definition}

In the definition above and in the lemma below,  
while $\tilde{x}, \tilde{y}$ may not exist for every $\omega \in \Omega$, rendering $d_n(x,y)=\infty$, the following states there is an event of high probability where this does not occur and where $d_n$ is moreover a metric. The following results from combining \cite[Def. 4.1 \& Prop. 4]{hoorn-2023} via \cite[Lemma 1]{hoorn-2023}:

\begin{lemma}
\label{approximating-properties}
    For all pairs of sequences $(\delta_n)_{n\in \mathbb{N}}$ and $(\varepsilon_n)_{n\in \mathbb{N}}$ satisfying \cref{delta-epsilon-assumption} and all positive numbers $Q$, there exists a sequence of events $(\Omega_n)_{n \in \mathbb{N}}$ and a sequence of real positive numbers $(c_n)_{n\in \mathbb{N}}$ such that $\lim_{n\rightarrow \infty} c_n=0$, $\mathbb{P}(\Omega_n)\geq 1-c_n \delta_n^3$, and the following properties hold for all $n\in \mathbb{N}$ and $\omega \in \Omega_n$:
    \begin{enumerate}
        \item $(B_{\delta_nQ}(x_0), d_n)$ is a metric space,
        \item for all $z_1,z_2 \in B_{\delta_nQ}(x_0):$ $|d_n(z_1,z_2)-d(z_1,z_2)| \leq c_n \delta_n^3$,
        \item for all $x,y \in B_{\delta_nQ}(x_0)$ there exists a path in $G_n \cup \{z_1,z_2\}$ connecting $z_1$ and $z_2$.
    \end{enumerate}
    Here $d_n$ is the random extended metric of \cref{extended-metric} and $G_n$ is the random graph of \cref{graphs-notation}, both being dependent on $\delta_n$ and $\varepsilon_n$.
\end{lemma}

\begin{remark}
\begin{itemize}
    \item The interpretation of the properties in \cref{approximating-properties} is that on the event $\Omega_n$ the random graph covers $B_{\delta_n Q}(x_0)$ in a way that allows to extend the graph distance to a distance on the manifold that closely approximates the original Riemannian distance with high probability.
    \item $Q$ can be interpreted as a localization scalar and can be an arbitrary finite number. For $Q >3$ it ensures in particular that if $z_1,z_2 \in B_{\delta_n}(x_0) \cup B_{\delta_n}(\exp_{x_0}(\delta_n v))$ then the geodesic connecting $z_1$ with $z_2$ lies in $B_{Q\delta_n}(x_0)$ by the triangle inequality, i.e. any point on the connecting geodesic lies in the "region of good approximation".
\end{itemize}    

\end{remark}

The events $(\Omega_n)_{n \in \mathbb{N}}$ for a weighted manifold are the same as in the uniform case. We now briefly outline their construction for completeness of the argument, although they will be treated as given from \cref{close-metrics} onwards.

The following is an adjustment of \cite[Lemma 2]{hoorn-2023} for non-uniform Poisson intensity. It is preliminary for the proof of \cref{cn-events}.

\begin{lemma}
\label{prelim-good-events} Let $(G_n(x_0,y_n,\delta_n))_{n \in \mathbb{N}}$ be the sequence of random geometric graphs of \cref{graphs-notation}.
Let $c>0$ and $(U_n)_{n\in \mathbb{N}}$ a sequence of finite subsets of $M$ such that $\# (U_n) = O(n^c)$ and let $(r_n)_{n\in \mathbb{N}}$ be a sequence with $r_n \sim \lambda_n$. Then there exists a sequence of positive reals $(c_n)_{n \in \mathbb{N}}$ such that $c_n \rightarrow 0$ and
$$
\mathbb{P} \left( \bigcup_{u \in U} \{B_{r_n}(u) \cap G_n = \emptyset\} \right) \leq c_n\delta_n^3 \quad \forall n \in \mathbb{N}.
$$
\end{lemma}

\begin{proof}
    By definition of the Poisson point process with intensity $n e^{-V(z)+V(x_0)}\vol(dz)$, the number of vertices of $G_n$ in the neighbourhood $B_{r_n}(u)$ is given by the Poisson random variable with mean
    $$
    \begin{aligned}
    n \int_{B_{r_n}(u)} e^{-V(z)+V(x_0)} \vol(dz) &= n e^{-V(u)+ V(x_0)} \int_{B_{r_n}(u)} e^{-V(z)+V(u)} \vol(dz) \\
    &= n e^{-V(u)+V(x_0)} \int_{B_{r_n}(u)} (1+O(r_n)) \vol(dz) \\
    &\sim n e^{-V(u)+V(x_0)} r_n^{\dim M} (1+O(r_n))\\
    &\sim n r_n^{\dim M} \\
    &\sim n \lambda_n^{\dim M}.
    \end{aligned}
    $$
    The term $e^{-V(z)+V(u)}$ with $z \in B_{r_n}(u)$ was bounded to an error or order $r_n$, and the factors which shrink to 1 as $n$ gets large were dropped due to the asymptotic relation $\sim$.

    On the penultimate line, we used the assumption that $M$ is compact, hence
    $$
    0 < \inf_{u \in M} e^{-V(u)} \leq \sup_{u \in M} e^{-V(u)} < \infty,
    $$
    so the factor $e^{-V(u)+V(x_0)}$ could be absorbed into the implicit proportionality constant. On the last line, we applied that $r_n \sim \lambda_n$. Then
    $$
    \begin{aligned}
    \mathbb{P}(B_{r_n}(u) \cap G_n = \emptyset) = \exp \left(n \int_{B_{r_n}(u)} e^{-V(z)+V(x_0)} \vol(dz) \right) = e^{-n \Theta(\lambda_n^{\dim M})},
    \end{aligned}
    $$
    and the rest of the argument follows as in \cite[Lemma 2]{hoorn-2023}.
\end{proof}

The following two lemmas together lead to specification of the events $\Omega_n$. They do not require adjustment for non-uniformity, hence for their proof we refer to \cite{hoorn-2023}. In particular, \cref{cn-events} can be obtained by an application of the previous lemma together with a covering argument on Riemannian manifolds.

\begin{lemma}
\label{cn-events}
\cite[Corollary 2]{hoorn-2023}
Let $(G_n(x_0,y_n,\delta_n))_{n \in \mathbb{N}}$ be the sequence of random geometric graphs of \cref{graphs-notation}. There exists a sequence of positive reals $(c_n)_{n\in \mathbb{N}}$ with $c_n \rightarrow 0$, a sequence of integers $(m_n)_{n \in \mathbb{N}}$ such that $m_n \sim \lambda_n^{-N}$ with $\lambda_n$ given by \eqref{lambda-sequence}, and for every $n \in \mathbb{N}$ a sequence of geodesic balls $B_{\lambda_n/4}(x_1), \ldots, B_{\lambda_n/4}(x_{m_n})$ covering a fixed compact neighbourhood of $x_0$ such that, denoting
$$
C_n := \bigcap_{i=1}^{m_n} \{B_{\lambda_n /4}(x_i) \cap G_n \neq \emptyset\},
$$
it holds that
$$
\mathbb{P}\left(C_n\right) \geq 1 - c_n\delta_n^3 \quad \forall n \in \mathbb{N}.
$$
\end{lemma}

The events $C_n$ warrant sufficient density of the vertices of the graphs, and thus finiteness of the extended metric $d_n$.

\begin{lemma}
\label{an-events}
\cite[Lemma 7]{hoorn-2023} 
Let $(G_n(x_0,y_n,\delta_n))_{n \in \mathbb{N}}$ be the sequence of random geometric graphs of \cref{graphs-notation} and let $Q >3$. There exists a sequence of positive reals $(c_n)_{n\in \mathbb{N}}$ with $c_n \rightarrow 0$ such that, denoting
$$
A_n:=\bigcap_{x,y \in B_{\delta_n Q}(x_0) \cap G_n} \{|d_{G_n}(x,y)-d(x,y)| \leq \f{3\lambda_n}{\varepsilon_n} d(x,y) + 2\lambda_n\},
$$
it holds that
$$
\mathbb{P}\left( A_n \right) \geq 1- c_n \delta_n^3 \quad \forall n \in \mathbb{N}.
$$
\end{lemma}

The events $A_n$ ensure that the extended metric $d_n$ matches closely the Riemannian metric $d$.
Then $\Omega_n := C_n \cap A_n$, $n\in \mathbb{N}$, are the events for \cref{approximating-properties}.

\begin{remark}
\begin{itemize}
\item The events $\Omega_n$ are thus constructed precisely to capture the Riemannian metric to a sufficient degree of approximation, while also granting sufficient density of graph vertices. These two properties are warranted by the events $A_n$ and $C_n$, respectively.
\item The non-uniform weight $e^{-V}$ only appears in \cref{prelim-good-events}, while the proofs of \cref{cn-events} and \cref{an-events} are unaffected by the weight. The non-uniformity is also important for construction of the sequence of random graphs $(G_n)_{n\in \mathbb{N}}$ in \cref{graphs-notation}, so that the corresponding empirical measures $\eta_{x_0}^{\delta_n}$ and $\eta_y^{\delta_n}$ approximate the non-uniform test measures $\nu_{x_0}^{\delta_n}$ and $\nu_y^{\delta_n}$, respectively, in the Wasserstein distance.

\item Instead of the Riemannian distance of connected vertices one could in general use any graph metric satisfying properties in \cref{approximating-properties}. Such graph metrics are referred to in \cite{hoorn-2023} as "$\delta_n$-good approximations" to the Riemannian metric.
\end{itemize}
\end{remark}

From \cite[Lemma 4]{hoorn-2023}, we quote:
\begin{lemma}
\label{close-metrics}
Let $d,\tilde{d}$ be two metrics on a set $X$ such that $(X,d)$, $(X,\tilde{d})$ are Polish spaces. If there exists a $C >0 $ such that
$$
\sup_{x_1 \in X, x_2 \in X} |d(x_1,x_2)-\tilde{d}(x_1,x_2)| \leq C,
$$
then for all probability measures $\mu,\nu$ on $X$:
$$
|W^{d}_1(\mu,\nu) - W^{\tilde{d}}_1(\mu,\nu)| \leq C.
$$
\end{lemma}

Denote by $W, W^{d_n}, W^{G_n}$ the Wasserstein distances corresponding to the Riemannian distance $d$, approximating manifold distance $d_n$, and the graph distance $d_{G_n}$, respectively.
As an immediate consequence of the preceding lemma and property (ii) in \cref{approximating-properties}, with $C=c_n\delta_n^3$, the Wasserstein distances $W^{d_n}$ and $W$ are close on $\Omega_n$:
\begin{corollary}
\label{tilde-wasserstein-approximation}
Given the sequence events $(\Omega_n)_{n\in \mathbb{N}}$ specified in \cref{approximating-properties}, there exists a sequence $(c_n)_{n \in \mathbb{N}}$ such that $c_n \rightarrow 0$ and for all probability measures $\mu,\nu$ on $B_{\delta_n Q}(x_0)$:
$$
\mathbb{E}\left[ |W_1(\mu,\nu) -W_1^{d_n}(\mu,\nu)| \Big| \Omega_n \right] \leq c_n \delta_n^3.
$$
\end{corollary}

\cref{tilde-wasserstein-approximation} provides an intermediate step in proving \cref{graph-continuum-limit-curvature}. To prove \cref{graph-continuum-limit-curvature},
we continue following along the lines of \cite{hoorn-2023}. On the event $\Omega_n$, the integrand in (\ref{G-M-distance}) can be split into three parts as
\begin{equation}
\label{splitting-wasserstein-distances}
\begin{aligned}
\left( W_1^{G_n}(\eta_{x_0}^{\delta_n}, \eta_{y_n}^{\delta_n})-W^{d_n}_1(\eta_{x_0}^{\delta_n}, \eta_{y_n}^{\delta_n}) \right)
&+ \left( W^{d_n}_1(\eta_{x_0}^{\delta_n}, \eta_{y_n}^{\delta_n})  - W^{d_n}_1(\nu_{x_0}^{\delta_n}, \nu_{y_n}^{\delta_n}) \right) \\
&+ \left( W^{d_n}_1(\nu_{x_0}^{\delta_n}, \nu_{y_n}^{\delta_n})-W_1(\nu_{x_0}^{\delta_n}, \nu_{y_n}^{\delta_n}) \right).
\end{aligned}
\end{equation}
The first term is 0 since the measures $\eta_{x_0}^{\delta_n}, \eta_{y_n}^{\delta_n}$ are supported on the nodes of the graphs $G_n$ and $d_n(x,y) = d_{G_n}(x,y)$ for all nodes $x,y \in \mathcal{V}_n$ as $d_n$ extends $d_{G_n}$ from the graph to the manifold. The last term is $o(\delta_n^3)$ on $\Omega_n$ by \cref{tilde-wasserstein-approximation} if we take arbitrary $Q>3$.
We can further estimate the second term in conditional expectation as
\begin{equation}
\label{eq:100}
\begin{aligned}
\E&\left[|W^{d_n}_1(\eta_{x_0}^{\delta_n}, \eta_{y_n}^{\delta_n})  - W^{d_n}_1(\nu_{x_0}^{\delta_n}, \nu_{y_n}^{\delta_n})| \big| \Omega_n \right] \\
&\hspace{2.5cm}\leq \E\left[ W^{d_n}_1(\eta_{x_0}^{\delta_n}, \nu_{x_0}^{\delta_n}) + W^{d_n}_1(\eta_{y_n}^{\delta_n}, \nu_{y_n}^{\delta_n}) \big| \Omega_n \right] \\
&\hspace{2.5cm} \leq \E\left[W_1(\eta_{x_0}^{\delta_n}, \nu_{x_0}^{\delta_n}) + W_1(\eta_{y_n}^{\delta_n}, \nu_{y_n}^{\delta_n}) \big| \Omega_n \right] + 2 c_n \delta_n^3.
\end{aligned}
\end{equation}
For some sequence $(c'_n)_{n \in \mathbb{N}}$ with $\lim_n c'_n =0$, the total expectation is therefore
\begin{equation}
\label{eq:101}
\begin{aligned}
&\E\left[\left| W_1^{G_n}(\eta_{x_0}^{\delta_n}, \eta_y^{\delta_n})-W_1(\nu_{x_0}^{\delta_n}, \nu_y^{\delta_n}) \right|\right] \\
&= \E\left[\left| W_1^{G_n}(\eta_{x_0}^{\delta_n}, \eta_y^{\delta_n})-W_1(\nu_{x_0}^{\delta_n}, \nu_y^{\delta_n}) \right| | \Omega_n\right] \mathbb{P}(\Omega_n) \\
& \quad + \E\left[\left| W_1^{G_n}(\eta_{x_0}^{\delta_n}, \eta_y^{\delta_n})-W_1(\nu_{x_0}^{\delta_n}, \nu_y^{\delta_n}) \right| |\Omega\setminus \Omega_n\right] (1-\mathbb{P}(\Omega_n))\\
&\leq \E[W_1(\eta_{x_0}^{\delta_n}, \nu_{x_0}^{\delta_n})|\Omega_n] + \E[W_1(\eta_{y}^{\delta_n}, \nu_{y}^{\delta_n})|\Omega_n] +c'_n\delta_n^3.
\end{aligned}
\end{equation}
We used that the second addend on the second line is uniformly bounded by $c_n\delta_n^3$ because $1-\mathbb{P}(\Omega_n) \leq c_n \delta_n^3$ by \cref{approximating-properties} and $\mathbb{P}$-almost surely
$$
W_1^{G_n}(\eta_{x_0}^{\delta_n}, \eta_y^{\delta_n}) \leq 3 \delta_n, \quad W_1(\nu_{x_0}^{\delta_n}, \nu_y^{\delta_n}) \leq 3 \delta_n.
$$
The first holds because the edge $(x_0,y)$ is always in the graph and every node in the support of $\eta_{x_0}^{\delta_n}$ (resp. $\eta_y^{\delta_n}$) is at most $\delta_n$ apart from $x_0$ (resp. $y$), hence $\textrm{supp } \eta_{x_0}^{\delta_n} \cup \textrm{ supp } \eta_y^{\delta_n} \subset B^{G_n}_{3\delta_n}(x_0)$. The second holds because $\textrm{supp } \nu_{x_0}^{\delta_n} \cup \textrm{ supp } \nu_y^{\delta_n} \subset B_{3\delta_n}(x_0)$ by Riemannian metric structure.

It remains to show that the terms on the right in (\ref{eq:101}) decay faster than $c''_n \delta_n^3$ as $n \rightarrow \infty$ for some sequence $(c''_n)_{n\in \mathbb{N}}$ with $\lim_n c''_n = 0$, which requires an extension of the original argument of \cite{hoorn-2023} to non-uniform measures. In particular, we show a generalization of Proposition 7 of the mentioned work, which is restated below in \cref{poisson-empirical-approximation}, and our generalization is \cref{nonuniform-wasserstein-matching}.

Our method consists in deforming uniform measures into non-uniform measures with small deviations from uniformity, and then showing the corresponding perturbance of distance is small enough, so that the change in the Wasserstein distance is also small.
The deformation maps we shall employ come from the next lemma.

For $U\subset \mathbb{R}^N$, an open subset, $k\in \mathbb{N}$, and $\alpha \in (0,1)$, denote by $C^{k,\alpha}(U)$ the Hölder space of functions with the norm
$$
\|g\|_{C^{k,\alpha}(U)}:= \sum_{i=0}^k \sup_{x\in U} \|\nabla^i g(x)\|+ \sup_{\substack{x,y\in U\\ x\neq y}} \frac{\|\nabla^k g(x)-\nabla^k g(y)\|}{\|x-y\|^{\alpha}}.
$$
Combining Theorems 1 and 2 of \cite{MR1046081} we note:
\begin{lemma}
\label{deformation-to-uniform}
Let $U \subset \R^N$ be an open set with a $C^{3+k,\alpha}$ boundary and $g \in C^{k,\alpha}(U)$ with $\int_{U} g(w) dw =0$. Then there exists a vector field $F \in C^{k+1,\alpha}(U,\R^N)$ such that the map $\psi: U \rightarrow U$ given by $\psi(x) =x+ F(x)$ is a diffeomorphism satisfying
\begin{equation}
\label{jacobian-pde}
\begin{aligned}
\det D\psi(x) &= 1+g(x) &\mathrm{in } \; U, \\
\psi(x)&=x &\mathrm{on } \; \partial U.
\end{aligned}
\end{equation}
Moreover, there is $C(U,k,\alpha)>0$ such that $\|F\|_{C^{k+1,\alpha}(U)} \leq C\|g\|_{C^{k,\alpha}(U)}$.
\end{lemma}

For our application, let $k=0$, $\alpha \in (0,1)$ arbitrary and $U=B_\delta(0)\subset \R^N$ with varying parameter $\delta >0$. 
We emphasize that in general the constant $C(U,k,\alpha)$ does depend on $U$, but for $U=B_\delta(0)$ one can deduce the following estimate which is uniform in $\delta$.
\begin{lemma}
There exists a $C>0$ such that for all $\delta >0$ and every $g \in C^1(\mathbb{R}^N)$ with $\int_{B_\delta(0)} g(z)dz =0 $ there exists $F \in C^1(B_\delta(0), \mathbb{R}^N)$ such that $\psi:B_\delta(0)\rightarrow B_\delta(0)$ given by $\psi(w) = w+ F(w)$ satisfies
\begin{equation}
\begin{aligned}
\det D\psi(w) &= 1+g(w) & \mathrm{in } \; B_\delta(0), \\
\psi(w)&=w & \mathrm{on } \; \partial B_\delta(0).
\end{aligned}
\end{equation}
Moreover,
\begin{equation}
\label{grad-F-bound}
\sup_{w \in B_\delta(0)} \|\nabla F(w)\| \leq C(\sup_{w \in B_\delta(0)} |g(w)| + \delta \sup_{w \in B_\delta(0)} \|\nabla g(w)\|).
\end{equation}
\end{lemma}

\begin{proof}
Let $C$ be the constant from \cref{deformation-to-uniform} for $U= B_1(0), k=0$ and any $\alpha \in (0,1)$ and consider the shrinking diffeomorphism 
$$
\phi_\delta: B_1(0) \rightarrow B_\delta(0), \quad
\phi_\delta(w) := \delta w.
$$
For any $g'\in C^{k,\alpha}(B_{\delta}(0))$ with $\int_{B_\delta(0)} g'(w) dw=0$, we have $g:= g' \circ \phi_\delta \in C^{k,\alpha}(B_1(0))$ and $\int_{B_1(0)} g'(\phi_\delta(w))dw =0$. Hence there exists $F \in C^{k+1,\alpha}(B_1(0),\R^N)$ such that $\psi(w)=w+F(w)$ satisfies (\ref{jacobian-pde}) with $g=g' \circ \phi_\delta$.
Define $F' := \delta F(\f{\cdot}{\delta})\in C^{k,\alpha}(B_\delta(0))$ and 
$$
\psi': B_\delta(0)\rightarrow B_\delta(0), \quad \psi'(w) := w+F'(w).
$$
Then
$$
D\psi'(\delta w)=D\psi(w) \quad \forall w \in B_1(0),
$$
and thus
$$
\det D\psi'(\delta w)=\det D\psi(w) = g'(\delta w) \quad \forall w \in B_1(0),
$$
which is equivalent to
$$
\det D\psi'(w) = g'(w) \quad \forall w\in B_\delta(0).
$$
Moreover,
$$
\begin{aligned}
\sup_{w \in B_\delta(0)} \|\nabla F'(w)\| &= \sup_{w \in B_1(0)}\|\nabla F(w)\| \\
&\leq C\left( \sup_{w\in B_1(0)} |g(w)|
+ \sup_{\substack{w_1,w_2 \in B_1(0)\\ w_1\neq w_2}} \f{|g(w_1)-g(w_2)|}{\|w_1-w_2\|^\alpha} \right)\\
&\leq C\left(\sup_{w\in B_1(0)} |g(w)| + \sup_{w\in B_1(0)} |\nabla g(w)|\right)\\
&= C \left(\sup_{w\in B_\delta(0)} |g'(w)| + \delta \sup_{w\in B_\delta(0)} |\nabla g'(w)|\right).
\end{aligned}
$$
\end{proof}

For any point $x\in M$ identify $T_xM \cong \R^N$. Assuming a uniform bound on $g_\delta$ and $\nabla g_\delta$, the following follows immediately from \eqref{grad-F-bound}.
\begin{corollary}
\label{vector-field-F}
Assume there are $C,C'>0$ such that for all $\delta >0$ and $x\in M$, $g_\delta \in C^1(\tilde{B}_\delta(x))$ satisfies $\int_{B_\delta(0)} g_\delta(w)dw =0 $ and
$$
\sup_{w\in \tilde{B}_\delta(x)} |g_\delta(w)|\leq C\delta, \quad
\sup_{w \in \tilde{B}_\delta(0)} \|\nabla g_\delta(w)\| \leq C'.
$$
Then for all $x\in M$ and $\delta>0$ there exists $F \in C^1(\tilde{B}_\delta(x), \R^N)$ such that $\psi:\tilde{B}_\delta(x)\rightarrow \tilde{B}_\delta(x)$ given by $\psi(w) = w+ F(w)$ satisfies
$$
\begin{aligned}
\det D\psi(w) &= 1+g_\delta(w) & \mathrm{in } \; \tilde{B}_\delta(x), \\
\psi(w)&=w & \mathrm{on } \; \partial \tilde{B}_\delta(x).
\end{aligned}
$$
Moreover, there exists $C''>0$ with
\begin{equation}
\label{F-derivative-bound}
\sup_{w \in \tilde{B}_\delta(x)} \|\nabla F(w)\| \leq C'' \delta \quad \forall x \in M, \delta >0.
\end{equation}
\end{corollary}

\begin{remark}
    \label{inverse-psi}
    Below we shall apply the preceding corollary in terms of the inverse $\psi^{-1}$ because of the change of variable formula for pushforwards, which brings in the inverse of the transformation map. That is, we will propose $g_\delta$ and find $\psi^{-1}(w) = w + \tilde{F}(w)$ such that
    $$
    \begin{aligned}
    \det D\psi^{-1}(w) &= 1+ g_\delta(w) & \textrm{in } \tilde{B}_\delta(x),\\
    \psi^{-1}(w) &= w & \textrm{on } \partial \tilde{B}_\delta(x)
    \end{aligned}
    $$
    and $\sup_{w \in \tilde{B}_\delta(x)} \| D\tilde{F}(w)\| = O(\delta)$.
    Then we can deduce by setting $w = \psi(\tilde{w})$ that 
    $$
    \psi(\tilde{w}) = \tilde{w} - \tilde{F}(\psi(\tilde{w})) = \tilde{w} + F(\tilde{w})
    $$
    for $F = -\tilde{F} \circ \psi$. Moreover,
    $$
    \begin{aligned}
    DF(\tilde{w}) &= -D\tilde{F}(\psi(\tilde{w})) D\psi(\tilde{w}) = -D\tilde{F}(w) (D\psi^{-1}(w))^{-1} \\
    & = -D\tilde{F}(w) (I + D\tilde{F}(w))^{-1} = O(\|D\tilde{F}(w)\|),
    \end{aligned}
    $$
    and hence $\sup_{\tilde{w} \in \tilde{B}_{\delta_n}(w)} \|DF(\tilde{w})\| = O(\delta)$ as well.

\end{remark}

Recall we denoted the uniform measure supported on the ball $\tilde{B}_{\delta}(x) \subset T_xM$ as
$$
d\tilde{\mu}^\delta_x(w) = \f{\one_{\tilde{B}_\delta(x)}(w)}
{|\tilde{B}_\delta(x)|} dw.
$$
Consider now the concrete sequence of functions $(g_{\delta_n})_{n \in \mathbb{N}}$ for our application given by 
\begin{equation}
\label{g-delta-def}
g_{\delta_n}(w) = \left(\f{d (\exp_x^{-1})_*\nu_x^{\delta_n}}{d \tilde{\mu}_x^{\delta_n}}( w)-1\right) \one_{\tilde{B}_{\delta_n}(x)}(w).
\end{equation}
Recalling \cref{lemma:density-mu-nu}, we can write
$$
\begin{aligned}
\f{d(\exp_x^{-1})_* \nu_x^{\delta_n}}{d \tilde{\mu}_x^{\delta_n}}(w)
= \f{d\nu_x^{\delta_n}}{d\bar{\mu}_x^{\delta_n}}(\exp_x w) &= \one_{\tilde{B}_{\delta_n}(x)}(w) \f{e^{-V(\exp_x w)}}{\int_{\tilde{B}_{\delta_n}(x)} e^{-V(\exp_x w')}d\tilde{\mu}(w')}.
\end{aligned}
$$
The function $e^{-V(\exp_x w)}$ is smooth in $w$ and does not depend on $\delta_n$, hence the normalisation constant satisfies
$$
\int_{\tilde{B}_{\delta_n}(x)} e^{-V(\exp_x w')}d\tilde{\mu}(w') = 1 + O(\delta_n^2)
$$
with a uniform constant over $n\in \mathbb{N}$. Hence
$$
\begin{aligned}
\sup_{w \in \tilde{B}_{\delta_n}(x)} |g_{\delta_n}(w)| &\leq (1 + O(\delta_n^2)) \sup_{w \in \tilde{B}_{\delta_n}(x)} |e^{-V(\exp_x w)}-1| = O(\delta_n), \\
\sup_{w \in \tilde{B}_{\delta_n}(x)} \|\nabla g_{\delta_n}(w)\| & \leq (1+O(\delta_n^2)) \sup_{w \in \tilde{B}_{\delta_n}(x)} \|\nabla e^{-V(\exp_x w)}\| \\
& \leq (1+O(\delta_n^2)) \sup_{w \in \tilde{B}_{\delta_n}(x)} e^{-V(\exp_x w)} \|\nabla V(\exp_x w)\| = O(1)
\end{aligned}
$$
uniformly for all $n \in \mathbb{N}$ because the expressions within the supremum again do not depend on $\delta_n$.
Hence the assumptions are satisfied for \cref{vector-field-F} to apply to  $g_\delta$.

\begin{corollary}
\label{rmk-pushforward}
The diffeomorphism $\psi: \tilde{B}_{\delta_n}(x) \rightarrow \tilde{B}_{\delta_n}(x)$ from \cref{inverse-psi}, for $g_{\delta_n}$ given by (\ref{g-delta-def}), satisfies
$(\exp_x)_* \psi_* \tilde{\mu}_x^{\delta_n} = \nu_x^{\delta_n}$.

\end{corollary}
\begin{proof}
For all $f \in \mathcal{B}_b(\tilde{B}_{\delta_n})$:
    $$
    \begin{aligned}
    \int f(w) d(\psi_* \tilde{\mu}_x^{\delta_n})(w) 
    &= \int f(\psi(w)) \ d\tilde{\mu}_x^{\delta_n}(w) \\
    &= \int f(w)|\det D \psi^{-1}(w)| \ d\tilde{\mu}_x^{\delta_n}(w) \\
    &= \int f(w) (1+g_{\delta_n}(w)) d\tilde{\mu}_x^{\delta_n}(w) \\
    &= \int f(w) \f{d(\exp_x^{-1})_* \nu_x^{\delta_n}}{d \tilde{\mu}_x^{\delta_n}} (w) d\tilde{\mu}_x^{\delta_n}(w) \\
    &= \int f(w) d((\exp_x^{-1})_* \nu_x^{\delta_n})(w),
    \end{aligned}
    $$
    applying the change of variable formula on the second line as $\tilde{\mu}_x^{\delta_n}$ is proportional to the Lebesgue measure, using the property of $\psi$ on the third line, and the definition of $g_{\delta_n}$ on the fourth line.
    This shows $\psi_* \tilde{\mu}_x^{\delta_n} = (\exp_x^{-1})_* \nu_x^{\delta_n}$,
    and pushing forward by the exponential mapping yields the result. 
\end{proof}

We now proceed with approximating the measures $\nu_x^{\delta_n}$ in optimal transport distance by the empirical measures coming from the Poisson point process $\mathcal{P}_n$, with the aim of finding an upper bound for (\ref{eq:100}). This was proved for Euclidean spaces in \cite[Prop. 7]{hoorn-2023}, building up on the work  of Talagrand \cite{MR1189420}.
\begin{lemma}
\label{poisson-empirical-approximation}
    Let $(b_n)_{n \in \mathbb{N}}$ be a sequence with $\lim_n b_n =0$ and let $\tilde{\eta}^{\delta_n}$ be the random measures on $B_{\delta_n}(0) \subset \R^N$ given by
    $$
    \tilde{\eta}^{\delta_n} = \frac{1}{N_0} \sum_{i=1}^{N_0} \delta_{X_i},
    $$
    where $N_0 \sim \mathrm{Poisson}\left(n(1+b_n)|B_{\delta_n}(0)|\right)$ and $X_i \sim \mathrm{Uniform}(B_{\delta_n}(0))$.
    Let $\mu^{\delta_n}$ be the uniform probability measure on $B_{\delta_n}(0)$. Then there is a sequence $(c_n)$ such that $\lim_n c_n =0 $ and 
    $$
    \E [W_1(\tilde{\eta}^{\delta_n}, \mu^{\delta_n})] \leq c_n \delta_n^3,
    $$
    where the $W_1$ distance is with respect to the standard Euclidean distance.

\end{lemma}

The following states that if the change of distances incurred by a bijection $\psi$ is small then the Wasserstein distances of any two measures pushed forward by $\psi$ may only increase a little.
\begin{lemma}
\label{distance-to-wasserstein}
    Let $\mathcal{X}, \mathcal{Y}$ be Polish spaces. For every $\delta >0$ let $\psi: \mathcal{X} \rightarrow \mathcal{Y}$ be a measurable map such that for all $K \subset M$ compact there is a $C(K) >0$ such that
    $$
    |d_\mathcal{Y}(\psi(x),\psi(y)) -d_\mathcal{X}(x,y)| \leq \delta C(K) d_\mathcal{X}(x,y) \quad \forall x,y \in K.
    $$
    Then for every $\mu,\nu \in P(K)$,
    $$
    W_1(\psi_*\mu, \psi_*\nu) \leq W_1(\mu,\nu)(1+C(K)\delta).
    $$
\end{lemma}

\begin{proof}
Let $\pi$ be an optimal coupling realising $W_1(\mu, \nu)$. Then the pushforward $\psi_*\pi := \pi(\psi^{-1}\cdot, \psi^{-1}\cdot)$ is a (not necessarily optimal) coupling between $\mu$ and $\nu$, therefore
$$
\begin{aligned}
W_1(\psi_*\mu, \psi_*\nu) &\leq \int_\mathcal{Y} d_\mathcal{Y}(x,y) d(\psi_* \pi)(x,y)
= \int_\mathcal{X} d_\mathcal{Y}(\psi(x), \psi(y)) d\pi(x,y)\\
&\leq \int_\mathcal{X} d_\mathcal{X}(x,y)(1+C(K)\delta)d\pi(x,y)
=W_1(\mu,\nu)(1+C(K)\delta).
\end{aligned}
$$
\end{proof}

\begin{proposition}
\label{nonuniform-wasserstein-matching}
There exists a sequence $(c_n)_{n \in \mathbb{N}}$ such that $\lim_n c_n = 0$ and
    $$
    \E[W_1(\eta_x^{\delta_n}, \nu_x^{\delta_n})] \leq c_n \delta_n^3 \quad \forall x \in B_{\delta_n}(x_0).
    $$
\end{proposition}
\begin{proof}
    For every $n \in \mathbb{N}$ and $x\in B_{\delta_n}(x_0)$ let $F \in C^1(\tilde{B}_{\delta_n}(x), \R^N)$ be the vector field such the map $\psi(w) =\exp_x (w+F(w))$ is a diffeomorphism such that $\det D_w \psi = 1+g_{\delta_n}(w)$ as per \cref{vector-field-F}. So in particular $(\exp_x)_* \psi_* \tilde{\mu}_x^{\delta_n} = \nu_x^{\delta_n}$ by \cref{rmk-pushforward}. Denote the normalizing constants by $Z_1(n)$ and $Z_2(n)$:
    $$
    Z_1(n) := \int_{\tilde{B}_{\delta_n}(x)} e^{-V(\exp_x w)+V(x)} dw, \quad Z_2(n):= \int_{B_{\delta_n}(x)} e^{-V(z)+V(x)} d\mathrm{vol}(z).
    $$
    Let $\tilde{\eta}_x^{\delta_n}$ be the normalized empirical measures of a Poisson point process with uniform (in $w$) intensity measure on $T_xM$ given by
    
    \begin{equation}
        \label{density-factor}
    \f{Z_2(n)}{Z_1(n)} ne^{-V(x)+V(x_0)} \one_{\tilde{B}_{\delta_n}}(w) dw = n(1+O(\delta_n))\one_{\tilde{B}_{\delta_n}}(w) dw
    \end{equation}
    as all of the factors $Z_1(n),Z_2(n)$ and $e^{-V(\exp_x w)+V(x_0)}$ are of the form $1+O(\delta_n)$.

    The form on the intensity measure as \eqref{density-factor} was chosen to make the following argument. Recall from \cref{notation:3} that
    $$
    d\tilde{\mu}_x^{\delta_n}(w) = \frac{1}{Z_1(n)} \one_{\tilde{B}_{\delta_n}(x)}(w) dw, \quad d\nu_x^{\delta_n}(z) = \f{e^{-V(z)+V(x)}}{Z_2(n)}d\mathrm{vol}(z).
    $$
    Then by \cref{ppp-pushforward}, $(\exp_x)_*\psi_* \tilde{\eta}_x^{\delta_n}$ are the normalized empirical measures of a Poisson point process with intensity
    $$
    \begin{aligned}
    (\exp_x)_* \psi_*\left(\f{Z_2(n)}{Z_1(n)} ne^{-V(x)+V(x_0)}\one_{\tilde{B}_{\delta_n}}(w) dw\right)&=Z_2(n) ne^{-V(x)+V(x_0)} d((\exp_x)_* \psi_*\tilde{\mu}_x^{\delta_n})(z)
    \\
    &=Z_2(n) ne^{-V(x)+V(x_0)} d\nu_x^{\delta_n}(z)\\
    &=Z_2(n) ne^{-V(x)+V(x_0)} \f{e^{-V(z)+V(x)}}{Z_2(n)}d\mathrm{vol}(z)\\
    &=ne^{-V(z)+V(x_0)}d\mathrm{vol}(z),
    \end{aligned}
    $$
    and hence $(\exp_x)_*\psi_* \tilde{\eta}_x^{\delta_n} = \eta_x^{\delta_n}$ in distribution.
    
    The distance upon applying $\psi$ satisfies the bound
    $$
    \begin{aligned}
    d(\psi(w_1),&\psi(w_2)) = \|w_1+F(w_1)-w_2-F(w_2)\| (1+O(\delta^2_n)) \\
    & \leq \left(\|w_1-w_2\|+ \int_0^1\|w_1-w_2\| \left\| \nabla F(w_1+t(w_2-w_1))\right\|dt \right)(1+O(\delta^2_n))\\
    &\leq\|w_1-w_2\|(1 + O(\sup_{t\in [0,1]} \|\nabla F(w_1+t(w_2-w_1))\|))(1+O(\delta_n^2)) \\
    &\leq \|w_1-w_2\|(1+C\delta_n),
    \end{aligned}
    $$
    using \cref{triangle-distance} on the first line and (\ref{F-derivative-bound}) on the last line.
    Therefore \cref{distance-to-wasserstein} applies and hence, together with \cref{poisson-empirical-approximation},
    $$
    \begin{aligned}
    \E[W_1(\eta_x^{\delta_n}, \nu_x^{\delta_n})] &= \E[W_1((\exp_x)_*\psi_* \tilde{\eta}_x^{\delta_n}, (\exp_x)_*\psi_* \tilde{\mu}_x^{\delta_n})] \\
    &\leq \E[W_1 (\tilde{\eta}_x^{\delta_n}, \tilde{\mu}_x^{\delta_n})](1+C\delta_n)\\
    &\leq c'_n \delta_n^3 (1+C\delta_n) \leq c_n \delta_n^3,
    \end{aligned}
    $$
    where $(c_n)_{n \in \mathbb{N}}$ is such that $c_n \rightarrow 0$ as $n \rightarrow \infty$.
\end{proof}

This proved the last missing piece, namely that (\ref{eq:101}) converges to 0 fast enough, thereby concluding the proof of \cref{graph-continuum-limit-curvature}.

\bibliographystyle{alpha}
\bibliography{OlliviersCoarseCurvature}

\appendix

\section{Appendix}
\begin{proof}[Proof of \cref{ricci-as-integral-over-sphere}]
The standard definition of Ricci curvature 
\begin{equation}
\label{eq:ricci-def}
\Ric(v,v) := \sum_{i=1}^n \la R(v, e_i)e_i, v \ra 
= \sum_{i=1}^n K(v,e_i) (1- \la v, e_i \ra^2)
\end{equation}
is independent of the choice of orthonormal basis $(e_i)$ of $T_pM$ \cite[Chap. 4.3]{MR3726907}. We exploit this by integrating the expression on the right over all orthogonal transformations in $SO(n)$ with respect to the Haar measure on $SO(n)$. Denote $S^{n-1}$ the unit sphere and note that it is a homogeneous space with the transitive group action 
$$
SO(n) \times S^{n-1} \rightarrow S^{n-1}, \quad (A,v) \mapsto Av,
$$
and moreover $S^{n-1} \cong SO(n) / SO(n-1)$.

Let $(e_1, \ldots e_n)$ be any orthonormal basis of $T_{x_0}M$. The average integral of the right-hand side of (\ref{eq:ricci-def}) over $SO(n)$ is
$$
\dashint_{SO(n)} \sum_{i=1}^n K(v,Ae_i) (1- \la v, Ae_i \ra^2) 
dA
= n \dashint_{S^{n-1}} K(v, w) (1- \la v, w \ra^2) d\sigma(w).
$$
The equality follows from the fact that the mapping 
$$
SO(n) \rightarrow S^{n-1}, A \mapsto Ae
$$
is surjective for any $e \in S^{n-1}$ and the pushforward of the Haar measure on $SO(n)$ is a multiple of the Riemannian volume measure on $S^{n-1}$, where we assume that $S^{n-1}$ is equipped with the homogeneous Riemannian structure. This multiplicative constant vanishes when taking the average integral and we obtain (\ref{eq:ricci-sphere-avg}) for $\varepsilon =1$. The change of variable $\tilde{w} = \varepsilon w$ gives the formula for arbitrary $\varepsilon >0$.
\end{proof}

\begin{proof}[Proof of \cref{lemma:ricci-average-ball}]
Let $\sigma$ denote the standard surface measure on $\partial B_r$ and write the integral on the right as
\begin{equation}
\begin{aligned}
\frac{1}{|B_\varepsilon|} &\int_0^\varepsilon \int_{\partial B_r}  K(v,w) (\|w\|^2 - \la v,w \ra^2) d\sigma(w) dr\\
&=  \int_0^\varepsilon \frac{\sigma(\partial B_r)}{|B_\varepsilon|} r^2 \dashint_{\partial
B_r} 
K(v,w)\lc 1- \frac{\la v,w \ra^2}{\|w\|^2} \rc d\sigma(w) dr\\
&= n \int_0^\varepsilon \varepsilon^{-n} r^{n+1} dr \dashint_{\partial B_1} K(v,w)(1-\la v,w \ra^2)
d\sigma(w)\\ &= \frac{\varepsilon^2}{n+2} \Ric(v,v).
\end{aligned}
\end{equation}
Here the second equation follows from the fact that the integrand of the $d\sigma(w)$ integral is independent of $r$, so the average can be taken over ball of any radius, and the integral over $r$ follows from the fact that
$$
|B_r|=\left(\f{r}{\varepsilon}\right)^n |B_\varepsilon|, \quad \sigma(\partial B_r) = \f{\partial |B_r|}{\partial r},
$$
which gives $\frac{\sigma(\partial B_r)}{|B_\varepsilon|}=n\varepsilon^{-n} r^{n-1}$. The last line follows from the identity (\ref{eq:ricci-sphere-avg}).
\end{proof}

\end{document}